\documentclass[reqno]{amsart}
\usepackage{graphicx, amsfonts,amssymb, amsmath,amsthm,color,mathtools,tikz,tikz-3dplot, bm, cancel, circledsteps}
\usepackage{amsaddr}

\definecolor{e-mail}{rgb}{0,.40,.80}
\definecolor{reference}{rgb}{.20,.60,.22}
\definecolor{citation}{rgb}{0,.40,.80}

\usepackage[pdfstartview=FitH, colorlinks, linkcolor=reference, citecolor=citation, urlcolor=e-mail]{hyperref}

\newtheorem{thm}{Theorem}
\newtheorem{lm}[thm]{Lemma}

\newtheorem{cnj}[thm]{Conjecture}
\theoremstyle{definition}

\newtheorem{cor}[thm]{Corollary}
\newtheorem{rmk}[thm]{Remark}
\numberwithin{thm}{section}
\numberwithin{equation}{section}

\title{Symmedians as Hyperbolic Barycenters}

\author{Maxim Arnold and Carlos E.~Arreche}
\address{Department of Mathematical Sciences, The University of Texas at Dallas, Richardson, TX 75080}
\email{Maxim.Arnold@utdallas.edu; Arreche@utdallas.edu}

\begin{document}

\begin{abstract}
The symmedian point of a triangle enjoys several geometric and optimality properties, which also serve to define it. We develop a new dynamical coordinatization of the symmedian, which naturally generalizes to other ideal hyperbolic polygons beyond triangles. We prove that in general this point still satisfies analogous geometric and optimality properties to those of the symmedian, making it into a hyperbolic barycenter. We initiate a study of moduli spaces of ideal polygons with fixed hyperbolic barycenter, and of some additional optimality properties of this point for harmonic (and sufficiently regular) ideal polygons.
\end{abstract}

\maketitle


\section{Introduction}
The \emph{symmedian point} of a triangle (also called its \emph{Lemoine point} or its \emph{Grebe point}) has been aptly called ``one of the crown jewels of modern geometry'' in \cite[\S7.1]{Honsberger}. Let us recall one way to construct it (cf.~\cite[\S7.4(iii)]{Honsberger}). Let $\mathbf{O}$ denote the circumcircle of 
$\Delta(ABC)$. Denote by $A^*$ the intersection of the tangents to $\mathbf{O}$ at the vertices $B$ and $C$, and define $B^*$ and $C^*$ analogously. Then the lines $AA^*$, $BB^*$ and $CC^*$ are concurrent at the \emph{symmedian point} $S$ of the triangle $ABC$.
\begin{figure}[!hbt]
    \centering
    \includegraphics[width=0.65\textwidth]{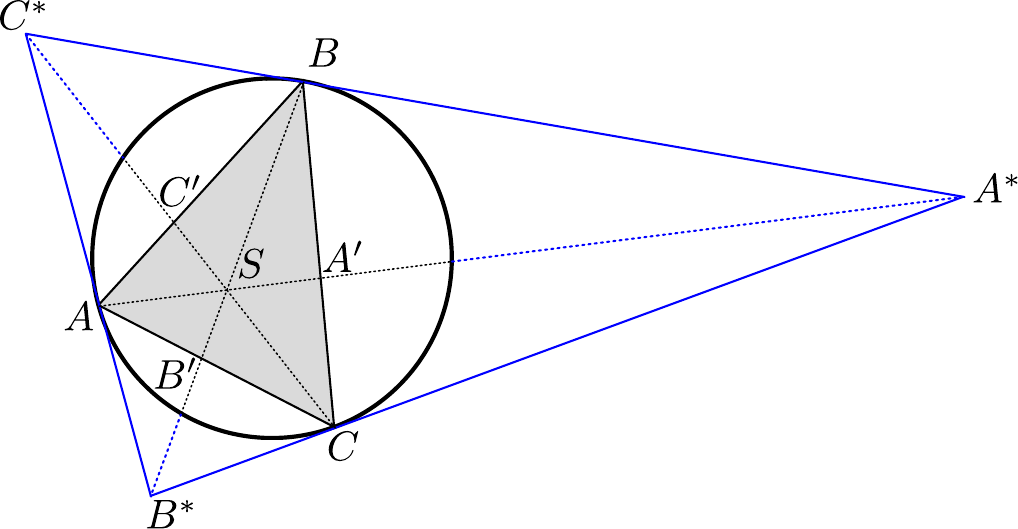}
    \caption{Construction of the symmedian point $S$ of the triangle $\Delta(ABC)$.}
    \label{fig:Sym}
\end{figure}

Among very many other properties enjoyed by the symmedian point $S$, there is the \mbox{following} optimization property, which is included as Exercise~7.3 in \cite{Honsberger}, and has been known (and reproved multiple times) since at least 1804 (cf.~\cite[p.~94]{Mackay}). 

\begin{thm}
\label{th:Euc}
The symmedian point $S$ of a triangle minimizes the sum of the squares of the distances from $S$ to the sides of the triangle. 
\end{thm}
This investigation originated from the humble observation that the above construction of the symmedian point of a triangle (usually proved as a theorem rather than used as a definition) makes it into the orthocenter of the ideal hyperbolic triangle with ``the same'' vertices, and how strange it seems that ``the same'' point should be meaningful from both the Euclidean and hyperbolic vantage points simultaneously. 
We shall show in Theorem~\ref{th:harmonic-coincidence} that this classical fact is only the first instance of a general property shared by all harmonic polygons beyond triangles.

The main goal of this short note is to introduce a generalization of the symmedian point $S$ of a triangle to inscribed polygons (or what is ``the same'', ideal hyperbolic polygons), which we call the \emph{hyperbolic barycenter}, and initiate a study of its properties. This point is uniquely defined in three different ways: by explicit coordinates~\eqref{eq:Sym}, by an optimality property (see Theorem~\ref{th:Klei}), and by a collection of geometric constructions (see \S\ref{sec:concurrence}), all directly analogous to those identifying the symmedian point of a triangle. We have the sense of having opened some long-forgotten Pandora's box, and 
we have now more questions than answers. We have striven to make this work as accessible as possible so that other researchers (including students!) may better address the questions that we leave unanswered. We hope the more sophisticated reader will not begrudge us our choice to keep the exposition, constructions, and proofs as elementary as we knew how.

Let us briefly describe the contents of this work in some more detail. In \S\ref{sec:elements}, after recalling some basic facts from hyperbolic geometry and fixing some conventions, we begin to refer to the \emph{Kleinization} 
of an inscribed Euclidean polygon, which is the ideal hyperbolic polygon obtained by identifying the open unit disc with the hyperbolic plane (in the Klein-Beltrami model), and to the opposite process of \emph{deKleinization} of hyperpolic geometric objects into Euclidean ones, and show in Theorem~\ref{th:cevian} the first of several hyperbolic minimality properties of the Kleinization of the symmedian point $S$ of a Euclidean triangle. 
In \S\ref{sec:coordinates} we introduce a dynamical coordinatization of this $S$, in terms of certain Hamiltonians introduced in \cite{AFIT} for the study of cross-ratio dynamics on ideal polygons, and use it to prove another minimality property in Theorem~\ref{th:Klei}, which makes $S$ into a hyperbolic barycenter.
In \S\ref{sec:polygons} we introduce the hyperbolic barycenters of general ideal hyperbolic polygons, and prove several interesting and useful properties that, in particular, allow for their several concurrent geometric (straightedge and compass) constructions. In \S\ref{sec:concurrence} we illustrate some of these geometric constructions for the first few ideal hyperbolic $n$-gons beyond triangles, with $n=4,6,10,5$. In \S\ref{sec:moduli} we initiate the study of the moduli spaces of ideal hyperbolic \mbox{$n$-gons} with a fixed common hyperbolic barycenter, and describe it explicitly for the smallest values of $n=3,4$ in terms of Poncelet conics. In \S\ref{sec:euclidean} we show that the classical Theorem~\ref{th:Euc} is the first instance of a more general Theorem~\ref{th:harmonic-coincidence} that holds uniformly for all \emph{harmonic} ideal hyperbolic polygons.
We prove a partial converse in Theorem~\ref{th:coincidence-4}, 
that the deKleinization of the hyperbolic barycenter of an ideal quadrilateral $\mathbf{P}$ minimizes the sum of squares of distances to the sides if and only if either $\mathbf{P}$ is harmonic or its deKleinization is a rectangle.

\section{Elements of hyperbolic geometry}\label{sec:elements}

In the \emph{Klein-Beltrami model}, the hyperbolic plane $\mathbb{H}$ is identified with the interior of the unit disc in $\mathbb{R}^2$. The unit circle boundary of this disc, called the \emph{absolute}, represents ``ideal points'' at infinite hyperbolic distance from any given point in $\mathbb{H}$. This absolute can be naturally identified with the real projective line $\mathbb{RP}^1$ (see e.g. \cite{Cox}), by associating with $p\in \mathbb{RP}^1$ the point \begin{equation}
    \label{eq:RP1}P=\left(\frac{1-p^2}{1+p^2},\frac{2p}{1+p^2}\right)
\end{equation}on the unit circle. Writing $p=\tan(\varphi/2)$ yields the usual parametrization $P=(\cos \varphi,\sin\varphi)$. The hyperbolic distance between two points $X,Y\in\mathbb{H}$ is defined by (see e.g. \cite[\S 19.2.5]{BergerII})
\begin{equation}\label{eq:point-dist}
d_{\mathcal{K}}(X,Y)=\cosh^{-1}{\dfrac{1-X\cdot Y}{\sqrt{(1-|X|^2)(1-|Y|^2|)}}}\end{equation} where $\cosh(z):=(e^z+e^{-z})/2$ denotes the hyperbolic cosine, $X\cdot Y$ stands for the usual inner product of two vectors in $\mathbb{R}^2$, and $|X|^2:=X\cdot X$. Geodesics are represented by straight lines, but the hyperbolic angles between geodesics do not in general agree with the corresponding Euclidean angles. Writing $\xi_1$ for the geodesic that intersects the absolute at $P_\ell=\left(\frac{1-p_\ell^2}{1+p_\ell^2},\frac{2p_\ell}{1+p_\ell^2}\right)$ for $\ell=1,2$ as in \eqref{eq:RP1}, the hyperbolic distance from $X=(x,y)\in\mathbb{H}$ to $\xi_1$ is computed by
\begin{equation}\label{eq:line-dist}
    d_\mathcal{K}(X,\xi_1)=\cosh^{-1}\sqrt{\frac{\left( x(1-p_1p_2) + y(p_1 + p_2)- 
   (1+p_1 p_2)\right)^2}{(p_1 - p_2)^2 (1 - x^2 - y^2)}+1}.
\end{equation}

The construction of the symmedian point of a triangle, described in Figure~\ref{fig:Sym}, is particularly suggestive to the hyperbolic geometer.
As mentioned in the introduction, if we interpret the circle $\mathbf{O}$ as the absolute in the Klein-Beltrami model, we can interpret the Euclidean triangle $\Delta(ABC)$ as an ideal hyperbolic triangle $\Delta_\mathcal{K}(ABC)$ in $\mathbb{H}$, with ``the same'' vertices. We call $\Delta_\mathcal{K}(ABC)$ the Kleinization of $\Delta(ABC)$, and $\Delta(ABC)$ the deKleinization of $\Delta_\mathcal{K}(ABC)$. We emphasize that the idea of ``Kleinizing'' Euclidean geometric objects is not new, see e.g. \cite{akop, RES}.

The point $A^*$ is called the \emph{polar} to the line $BC$ with respect to $\mathbf{O}$, and any geodesic through $A^*$ intersects $BC$ orthogonally in $\mathbb{H}$. Therefore, the symmedian point of $\Delta(ABC)$ is the point of intersection of the hyperbolic altitudes of $\Delta_\mathcal{K}(ABC)$, i.e., its hyperbolic orthocenter. This simple observation readily provides another minimality property of the symmedian point $S$, as follows. For any Euclidean triangle $\Delta(ABC)$ and a given point $X$ in the plane, consider its \emph{Cevian triangle} $\Delta(A'B'C')$, where $A'$ is the intersection of the line $AX$ with the line $BC$, and the points $B'$ and $C'$ are defined similarly.

\begin{thm}\label{th:cevian}
The symmedian point of a Euclidean triangle minimizes the hyperbolic perimeter of the Kleinization of its Cevian triangle.
\end{thm}
\vspace{-.05in}
\begin{figure}[!hbt]
    \centering
    \includegraphics[width=0.37\textwidth]{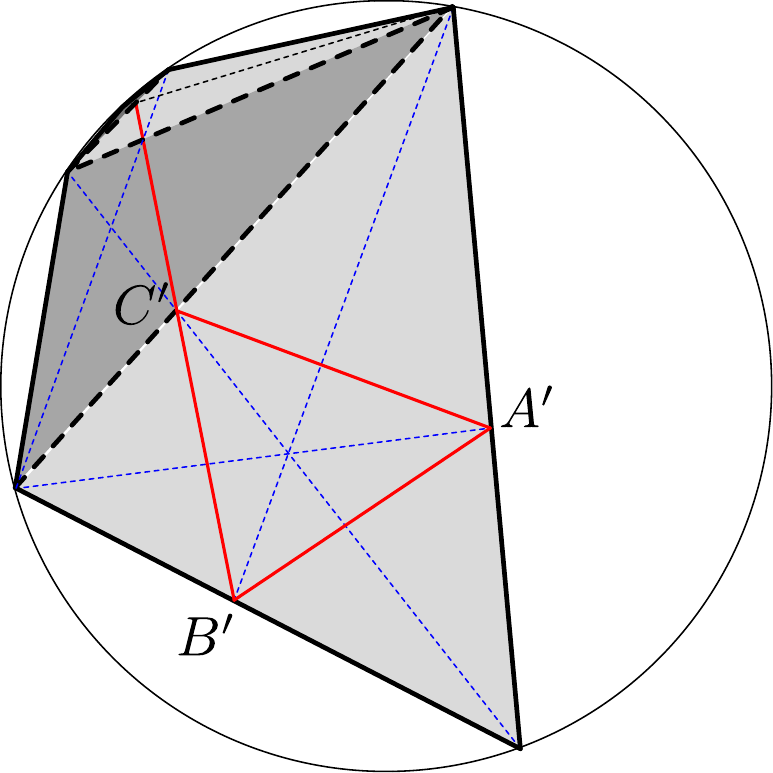}\qquad\qquad
    \includegraphics[width=0.37\textwidth]{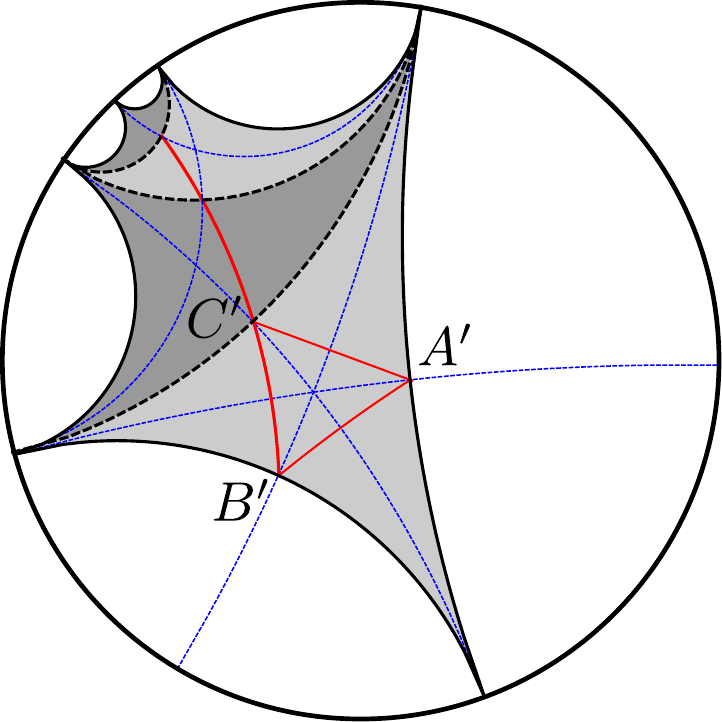}
    \caption{Another minimality property of the Symmedian point, from both the Klein-Beltrami (left) and the Poincar\'e (right) points of view.}
    \label{fig:minimal}
\end{figure}

\begin{proof}
This follows from the classic reflection argument about Fargano's orbit: under sequential hyperbolic reflections of the triangle $ABC$ in its sides, the sides of the triangle $A'B'C'$ form a geodesic interval (see Figure \ref{fig:minimal}).
\end{proof}

\section{Coordinates and hyperbolic minimality property of the symmedian}\label{sec:coordinates}

For an ideal polygon $\mathbf{P}=(p_1,\dots p_n)$ the following three quantities were introduced in \cite[\S 6.2]{AFIT}, and shown to be Hamiltonians for the infinitesimal diagonal action of 
the M\"obius group.
\begin{equation}\label{eq:hamiltonian}I_{\mathbf{P}}=\sum_{\ell=1}^n \dfrac{1}{p_{\ell}-p_{\ell+1}},\qquad J_{\mathbf{P}}=\frac 12\sum_{\ell=1}^n \dfrac{p_{\ell}+p_{\ell+1}}{p_{\ell}-p_{\ell+1}},\qquad  K_{\mathbf{P}}=\sum_{\ell=1}^n \dfrac{p_\ell p_{\ell+1}}{p_{\ell}-p_{\ell+1}} \end{equation} where indices are understood modulo $n$, that is, with $p_{n+1}:=p_1$.

\begin{thm}
The symmedian point $S_\mathbf{P}$ of a triangle $\mathbf{P}=(p_1,p_2,p_3)$ inscribed in the unit circle is 
\begin{equation}\label{eq:Sym}S_\mathbf{P}=\left(\dfrac{I_{\mathbf{P}}-K_{\mathbf{P}}}{I_{\mathbf{P}}+K_{\mathbf{P}}}, \dfrac{2J_{\mathbf{P}}}{I_{\mathbf{P}}+K_{\mathbf{P}}}\right).\end{equation}
\end{thm}

\begin{proof}
    This is an elementary computation. 
    The polar $P_\ell^*=(a_\ell^*,b_\ell^*)$ to the side $\xi_\ell$ is given by
      \begin{equation}\label{eq:polar}
        P_{\ell}^*=(a_\ell^*,b_\ell^*)= \left(\dfrac{1-p_{\ell}p_{\ell+1}}{1 + p_{\ell} p_{\ell+1}},\dfrac{p_{\ell} + p_{\ell+1}}{1 + p_{\ell} p_{\ell+1}}\right).
    \end{equation}
 Writing $P_\ell=(a_\ell,b_\ell)$ as in \eqref{eq:RP1} and $S_\mathbf{P}=(\alpha,\beta)$ as in \eqref{eq:Sym}, we compute from \eqref{eq:polar} that \[\frac{\beta-b_{\ell-1}^*}{\alpha-a_{\ell-1}^*}=\frac{b_{\ell-1}^*-b_\ell^*}{a_{\ell-1}^*-a_\ell^*}=\frac{b_\ell^*-\beta}{a_\ell^*-\alpha}.\] Hence $S_\mathbf{P}$ as given in \eqref{eq:Sym} is the common intersection of the (at least two) lines $P_{\ell-1} P_\ell^*$ such that $p_{\ell}p_{\ell+1}\neq -1$, so it is the symmedian of $\mathbf{P}$.\qedhere
\end{proof}
 
We can now show a second minimality property of the Kleinization of $S_\mathbf{P}$, which makes it into a kind of ``hyperbolic barycenter''.

\begin{thm}\label{th:Klei}
    The point $S=S_\mathbf{P}$ in \eqref{eq:Sym} minimizes the sum of the hyperbolic sines of the hyperbolic distances from $S$ to the sides of the Kleinization of $\, \mathbf{P}$.
\end{thm}

\begin{proof}
    Using \eqref{eq:line-dist} together with the identity $\cosh^{-1}(t)=\sinh^{-1}(\sqrt{t^2-1})$ yields, for the hyperbolic distance $d_\mathcal{K}(X,\xi_\ell)$ from a point $X=(x,y)$ to the side $\xi_\ell=P_{\ell} P_{\ell+1}=\xi(p_\ell,p_{\ell +1})$ (where we once again consider indices modulo $n$ as in \eqref{eq:hamiltonian}), that
    \begin{equation}\label{eq:quadrancy}
        q_\ell(X):=\sinh\bigl(d_\mathcal{K}(X,\xi_\ell)\bigr)=
        \dfrac{((I_\ell-K_\ell) x +2 J_\ell y - (I_\ell+K_\ell))}{\sqrt{1-(x^2+y^2)}}
    \end{equation}
    where $I_{\ell}:=\dfrac{1}{p_{\ell}-p_{\ell+1}}$; $2J_{\ell}:=\dfrac{p_{\ell}+p_{\ell+1}}{p_{\ell}-p_{\ell+1}}$; and $K_{\ell}:=\dfrac{p_{\ell} p_{\ell+1}}{p_{\ell}-p_{\ell+1}}$; so that \begin{equation}\label{eq:IJK-sums}\textstyle I:=\sum_\ell I_\ell, \qquad J:=\sum_\ell J_\ell, \qquad K:=\sum_\ell K_\ell\end{equation} coincide with the expressions in \eqref{eq:hamiltonian}. Writing $\mathcal{Q}(X):=\sum_\ell q_\ell(X)$, we compute the gradient
        \begin{equation}\label{eq:gradient}
        \nabla(\mathcal{Q}(X)) = 
        (1-|X|^2)^{-\frac{3}{2}}
        \begin{pmatrix}
            \Bigl((I-K)y - (2J) x\Bigr)y +(I+K)x-(I-K)\vphantom{\displaystyle \sum_i}\\
            \Bigl((2J)x-(I-K)y\Bigr) x+(I+K)y-(2J)
        \end{pmatrix}.
    \end{equation}
    We claim that the only point $X$ with $|X|\neq 1$ that nullifies the polynomial part of the gradient $(1-|X|^2)^{\frac{3}{2}}\nabla(\mathcal{Q}(X))$ is $X=S$. Denoting by $W^\perp=(-z,w)$ the counterclockwise rotation of \mbox{$W=(w,z)$} by $\pi/2$,
    setting $\nabla(\mathcal{Q}(X)) =0$ in \eqref{eq:gradient} is equivalent to
    \begin{equation}\label{eq:better-gradient}-(X\cdot S^\perp) X^\perp+X-S=(0,0).\end{equation} 
    For $X$ any solution to \eqref{eq:better-gradient}, taking the dot product of \eqref{eq:better-gradient} with $X$ yields $|X|^2=S\cdot X$, which implies that either $X=(0,0)$ (in which case $S=(0,0)=X$, by \eqref{eq:better-gradient}, as desired), or else $S=X+sX^\perp$ for some $s\in\mathbb{R}$, and therefore $S^\perp=X^\perp-sX$. Then \eqref{eq:better-gradient} becomes $s(|X|^2-1)X^\perp=(0,0).$ This further implies that either (i) $s=0$; or (ii) $|X|^2=1$; or (iii) $X^\perp=(0,0)$. In case (i) we find that $X=S$; case (ii) is forbidden; 
    and in case (iii) we again find that $X=(0,0)=S$. Thus $X=S$ is the only point  not in $\mathbf{O}$ satisfying \eqref{eq:better-gradient}. 
    To conclude, observe that the Hessian determinant of $\mathcal{Q}(X)$ at $S$ is $\frac{(I+K)^2}{(1-|S|^2)^2}>0.$\qedhere
    \end{proof}

\section{Hyperbolic barycenters of ideal polygons}\label{sec:polygons}

The statement and the proof of Theorem~\ref{th:Klei} both generalize immediately to ideal hyperbolic polygons $\mathbf{P}$ beyond triangles.
Since $S_\mathbf{P}$ in \eqref{eq:Sym} minimizes the sum of the hyperbolic sines of the hyperbolic distances to the sides of $\mathbf{P}$, we find it natural to call it the \emph{hyperbolic barycenter} of $\mathbf{P}$.

     \begin{rmk} \label{rm:positive} It is easy to see that  $I_\mathbf{P}+K_{\mathbf{P}}\neq 0$, not just generically, but always when $\mathbf{P}$ is convex. Under the parameterization of \S\ref{sec:elements}, each $I_\ell+K_\ell=\frac{1+p_{\ell}p_{\ell+1}}{p_{\ell}-p_{\ell+1}}=\cot\left(\frac{\varphi_\ell-\varphi_{\ell+1}}{2}\right)$. At most one $I_\ell+K_\ell$ has different sign from the others (only in case $|\varphi_\ell -\varphi_{\ell+1}|\geq \pi$). Even in this exceptional case, since $|\varphi_{\ell'}-\varphi_{\ell'+1}|<2\pi-|\varphi_{\ell}-\varphi_{\ell+1}|$ for any other $\ell'\neq\ell$, we obtain immediately that $|I_{\ell'}+K_{\ell'}|>|I_{\ell}+K_{\ell}|$.  
    \end{rmk}

 For $\Xi=\{\xi_{\ell_1},\dots,\xi_{\ell_k}\}$ a non-empty set of (not necessarily contiguous) sides of $\mathbf{P}$, define 
 \[I_{\Xi}:=\sum_{\xi_\ell\in\Xi} I_\ell;\qquad J_{\Xi}:=\sum_{\xi_\ell\in\Xi} J_\ell;\qquad K_{\Xi}:=\sum_{\xi_\ell\in\Xi} K_\ell\] analogously as in \eqref{eq:IJK-sums}. We similarly let, with $q_\ell(X)=\sinh{\bigl(d_\mathcal{K}(X,\xi_\ell)\bigr)}$ as in \eqref{eq:quadrancy}, \begin{equation}\label{eq:xi-notation}\mathcal{Q}_\Xi(X)=\sum_{\xi_\ell\in\Xi}q_\ell(X)\qquad\text{and}\qquad S_\Xi=\left(\frac{I_\Xi-K_\Xi}{I_\Xi + K_\Xi},\frac{2J_\Xi}{I_\Xi + K_\Xi}\right).\end{equation} In case $\Xi$ is the set of all sides of $\mathbf{P}$, we still simply write $S_\mathbf{P}$ (as in \eqref{eq:Sym}) instead of $S_\Xi$, $\mathcal{Q}_\mathbf{P}$ instead of $\mathcal{Q}_\Xi$, etc. The computations in the proof of Theorem~\ref{th:Klei} immediately yield the following.

\begin{cor} \label{cor:linear-gradient}
    Let $\Xi=\{\xi_{\ell_1},\dots,\xi_{\ell_k}\}$ be a non-empty set of sides. Then $S_\Xi$ nullifies the polynomial part of the gradient $(1-|X|^2)^{\frac{3}{2}}\nabla(\mathcal{Q}_\Xi)$ as in \eqref{eq:gradient}. \end{cor}
  
    \begin{rmk}\label{rm:twosidesminimizer}
        The point $S_\Xi$ need not be in $\mathbb{H}$. For example, $S_{\{\xi\}}$ is the polar to $\xi$, which is outside of the unit disk, and $S_{\{\xi_1,\xi_2\}}=P_2 
        $, 
        the common ideal vertex 
        of the sides $\xi_1$ and $\xi_2$, which lies on $\mathbf{O}$. 
    \end{rmk}

    We set aside a more systematic discussion of the geometric meaning of $S_\Xi$ for arbitrary $\Xi$, 
    in favor of focusing on the most immediately relevant configuration properties of the points $S_\Xi$ for different interrelated choices of $\Xi$.

\begin{lm}[Interpolation Lemma]\label{lm:interpolation}
    Let $\boldsymbol{\Xi}=\{\Xi_1,\dots,\Xi_\nu\}$, where the $\Xi_j$ are (not necessarily disjoint) sets of sides of $\mathbf{P}$, such that $\sum_{j=1}^\nu(I_{\Xi_j}+K_{\Xi_j})\neq 0$. Denote $\mathcal{Q}_{\boldsymbol{\Xi}}(X):=\sum_{j=1}^\nu\mathcal{Q}_{\Xi_j}(X)$ and 
    \begin{equation}\label{eq:sxialphabeta} S_{\boldsymbol{\Xi}}:=\left(\sum_{j=1}^\nu (I_{\Xi_j}+K_{\Xi_j})\right)^{-1}\sum_{j=1}^\nu(I_{\Xi_j}+K_{\Xi_j})\cdot S_{\Xi_j}.\end{equation} Then $S_{\boldsymbol{\Xi}}$ nullifies the polynomial part of the gradient  $(1-|X|^2)^{\frac{3}{2}}\nabla(\mathcal{Q}_{\boldsymbol{\Xi}})$.
\end{lm}

\begin{proof} As in the proof of Theorem~\ref{th:Klei}, we see that $X$ nullifying the polynomial part of the gradient $(1-|X|^2)^{\frac{3}{2}}\nabla(\mathcal{Q}_{\boldsymbol{\Xi}})$ is equivalent to $X$ satisfying \eqref{eq:better-gradient} with $S=S_{\boldsymbol{\Xi}}$ as in \eqref{eq:sxialphabeta}. \qedhere
\end{proof}

\begin{cor} \label{rm:metathm}
The points $S_{\{\Xi_1,\Xi_2\}}$, $S_{\Xi_1}$, and $S_{\Xi_2}$ are collinear. Moreover, if $\mathbf{P}$ is convex and each $|\Xi_j|\geq 2$ then $S_{\boldsymbol{\Xi}}$ is in the convex hull of the $S_{\Xi_j}$. 
\end{cor}

\begin{proof}
    In \eqref{eq:sxialphabeta}, $S_{\boldsymbol{\Xi}}$ is a weighted sum of the $S_{\Xi_j}$ with weights summing to $1$. Moreover, if $\mathbf{P}$ is convex and each $|\Xi_j|\geq 2$ then the $I_{\Xi_j} + K_{\Xi_j}$ all have the same sign (cf.~Remark~\ref{rm:positive}), so these weights are positive. 
\end{proof}

     Corollary~\ref{rm:metathm} is one of the main tools that allows for the geometric construction of hyperbolic barycenters. 
     It is educational to see it in action in some trivial cases.

     For any ideal $n$-gon $\mathbf{P}$, we can take $\boldsymbol{\Xi}=\{\Xi_1,\dots,\Xi_n\}$, where $\Xi_\ell=\{\xi_{\ell-1},\xi_\ell\}$, the two sides that meet at the ideal vertex $P_\ell$.
     Then $\mathcal{Q}_{\boldsymbol{\Xi}}(X)=2\mathcal{Q}_\mathbf{P}(X)$ and each $S_{\Xi_\ell}=P_\ell$ (cf.~Remark~\ref{rm:twosidesminimizer}). Thus if $\mathbf{P}$ is convex then and $S_{\boldsymbol{\Xi}}=S_\mathbf{P}$ is not only in $\mathbb{H}$, but actually in the interior of $\mathbf{P}$ by Corollary~\ref{rm:metathm}, which is neither surprising nor immediately apparent from the algebraic \mbox{expression \eqref{eq:Sym}.}
     
     In the tired case of triangles $\mathbf{P}=(p_1,p_2,p_3)$, letting $\Xi_1=\{\xi_\ell\}$ and $\Xi_2=\{\xi_{\ell-1},\xi_{\ell+12}\}$, we have $\mathcal{Q}_\mathbf{P}(X)=\mathcal{Q}_{\Xi_1}(X)+\mathcal{Q}_{\Xi_2}(X)$, $S_{\Xi_1}$ is the polar $P_\ell^*$ to $\xi_\ell$, and $S_{\Xi_2}=P_{\ell-1}$ (cf.~Remark~\ref{rm:twosidesminimizer}). Thus Corollary~\ref{rm:metathm} says in this case that $S_\mathbf{P}=S_{\{\Xi_1,\Xi_2\}}$ lies on $P_{\ell+2} P_\ell^*$, as expected. A first non-trivial consequence of Corollary~\ref{rm:metathm} is that this argument works all the same even if the first and third sides do not meet to form a triangle, as formalized in the following result.

\begin{lm}\label{lm:three-sides-minimizer}
    Let $\Xi=\{\xi_{\ell-1},\xi_{\ell},\xi_{\ell+1}\}$ be any set of three consecutive sides. Then $S_{\Xi}$ is the intersection of the two hyperbolic altitudes, from $P_{\ell}$ to $\xi_{\ell+1}$, and from $P_{\ell+1}$ to $\xi_{\ell-1}$.
\end{lm}

\begin{proof}
    Partition $\Xi$ in two different ways: $\boldsymbol{\Xi}_1=\{\Xi_{11},\Xi_{12}\}$ and $\boldsymbol{\Xi}_2=\{\Xi_{21},\Xi_{22}\}$, where $\Xi_{11}=\{\xi_{\ell-1}\}$ and $\Xi_{12}=\{\xi_{\ell},\xi_{\ell+1}\}$, and $\Xi_{21}=\{\xi_{\ell+1}\}$ and $\Xi_{22}=\{\xi_{\ell-1},\xi_{\ell}\}$. Then $S_{\boldsymbol{\Xi}_1}=S_\Xi=S_{\boldsymbol{\Xi}_2}$ by Lemma~\ref{lm:interpolation}, whence $S_\Xi$ lies on both geodesic segments $S_{\Xi_{11}}S_{\Xi_{12}}$ and $S_{\Xi_{21}}S_{\Xi_{22}}$ by Corollary~\ref{rm:metathm}. To conclude, we observe as in Remark~\ref{rm:twosidesminimizer} that $S_{\Xi_{11}}=P_{\ell-1}^*$ and $S_{\Xi_{12}}=P_{\ell+1}$, and $S_{\Xi_{21}}=P_{\ell+1}^*$ and $S_{\Xi_{22}}=P_{\ell}$.
\end{proof}

\begin{rmk}
    Lemma~\ref{lm:three-sides-minimizer} is the third in a recursive series of constructions initiated in Remark~\ref{rm:twosidesminimizer}. For $\Xi=\{\xi_0,\dots,\xi_\ell\}$ a set of $\ell+1$ consecutive sides, we find that $S_\Xi$ is the intersection of $P_0^* S_{\Xi_0}$ and $S_{\Xi_\ell}P_\ell^*$, where $\Xi_m:=\Xi-\{\xi_m\}$. More generally, the geodesic segments $S_{\{\xi_0,\dots,\xi_{m-1}\}}S_{\{\xi_{m},\dots,\xi_{\ell}\}}$ for $m=1,\dots,\ell$ are all concurrent at $S_\Xi$.
\end{rmk}

As in \cite[\S2.1]{AFIT}, we say that two ideal $n$-gons $\mathbf{P}$ and $\mathbf{Q}$ are $\alpha$\emph{-related} if the cross-ratio $[p_\ell,q_\ell,p_{\ell+1},q_{\ell+1}]=\alpha$ for every $\ell$, with the convention that the polygons have mutually orthogonal sides when $\alpha=-1$. Note that $\mathbf{P}$ and $\mathbf{Q}$ are $\alpha$-related (for some $\alpha$) if and only if the hyperbolic angles between the corresponding sides of $\mathbf{P}$ and $\mathbf{Q}$ are all equal.  

\begin{thm}\label{th:alpha-related}
If two ideal polygons $\mathbf{P}$ and $\mathbf{Q}$ are $\alpha$-related for some $\alpha\in\mathbb{R}$, then $S_\mathbf{P}=S_\mathbf{Q}$.
\end{thm}

\begin{proof}
    By \cite[Thm.~16]{AFIT}, the values of $I,J,K$ in \eqref{eq:hamiltonian} coincide for $\alpha$-related polygons. \qedhere
\end{proof}

\begin{rmk} \label{rm:alpha-related} Theorem~\ref{th:alpha-related} can also be seen in \cite{AFIT} from a complementary point of view. The barycenter of an ideal hyperbolic polygon is \emph{defined} in \cite[\S6.2.1]{AFIT} to be the limit as $\varepsilon\rightarrow 0$ of the fixed point of the composition $L_{1+\varepsilon}(\mathbf{P})$ of infinitesimal rotations along the sides of the polygon $\mathbf{P}$ -- let us temporarily denote this point $\tilde{S}_\mathbf{P}$. As explained in \cite[Rmk.~6.9]{AFIT}, the M\"obius invariance of $IK-J^2$ (cf.~\cite[Lem.~6.6]{AFIT}) implies $\tilde{S}_\mathbf{P}=\tilde{S}_\mathbf{Q}$ for $\alpha$-related $\mathbf{P}$ and $\mathbf{Q}$. Then in \cite[Lem.~6.10]{AFIT} it is shown that $\tilde{S}_\mathbf{P}$ minimizes 
$\mathcal{Q}_\mathbf{P}(X)$ given in \eqref{eq:xi-notation}, thus \emph{a posteriori} $\tilde{S}_\mathbf{P}=S_\mathbf{P}$ as in \eqref{eq:Sym}, by Theorem~\ref{th:Klei}.\end{rmk}

In the Euclidean setting, there is an analogous notion to the cross-ratio dynamics considered in \cite{AFIT} called \emph{bicycle correspondence}, which also has a fixed point called \emph{circumcenter of mass} (cf.~\cite[\S1, Thm.~5, Rmk.~3.3]{CCM1}). The next property of the hyperbolic barycenter $S_\mathbf{P}$ resembles the similar Archimedian property of the Euclidean circumcenter of mass \cite[Thm.~1]{CCM2}. 

\begin{lm}[Archimedean property] \label{lm:arch}
If the ideal polygon $\mathbf{P}$ is partitioned along any of its diagonals into ideal polygons $\mathbf{Q}_1$ and $\mathbf{Q}_2$ , then
\begin{equation}\label{eq:Archimed}
S_\mathbf{P}=\frac{I_{\mathbf{Q}_1}+K_{\mathbf{Q}_1}}{I_{\mathbf{P}}+K_{\mathbf{P}}}S_{\mathbf{Q}_1}+
\frac{I_{\mathbf{Q}_2}+K_{\mathbf{Q}_2}}{I_{\mathbf{P}}+K_{\mathbf{P}}}S_{\mathbf{Q}_2}.
\end{equation}
\end{lm}

\begin{proof}
Indeed, let $\mathbf{Q}_1=(p_1,\dots,p_t)$ and $\mathbf{Q}_2=(p_1,p_{t},\dots,p_n)$. One has $$I_{\mathbf{P}}=
\left(\sum\limits_{\ell=1}^{t-1} \dfrac{1}{p_{\ell}-p_{\ell+1}} +\dfrac{1}{p_t-p_1}\right)+\left(\dfrac{1}{p_1-p_t}+\sum\limits_{\ell=t}^{n} \dfrac{1}{p_{\ell}-p_{\ell+1}}\right)=I_{\mathbf{Q}_1}+I_{\mathbf{Q}_2}.$$
Similarly, $J_{\mathbf{P}}=J_{\mathbf{Q}_1}+J_{\mathbf{Q}_2}$ and $K_{\mathbf{P}}=K_{\mathbf{Q}_1}+K_{\mathbf{Q}_2}$, and \eqref{eq:Archimed} follows from \eqref{eq:Sym}.
\end{proof}

\begin{cor}\label{cor:archimedian}
    The points $S_\mathbf{P}$, $S_{\mathbf{Q}_1}$, and $S_{\mathbf{Q}_2}$ are collinear.
\end{cor}

\begin{proof}
    In \eqref{eq:Archimed}, $S_{\mathbf{P}}$ is a weighted sum of $S_{\mathbf{Q}_1}$ and $S_{\mathbf{Q}_2}$ with weights summing to $1$.
\end{proof}

\begin{rmk}
Let us briefly compare the hyperbolic barycenter $S_\mathbf{P}$ with other centers defined by integrable polygonal dynamics. As already mentioned, the cross-ratio dynamics of \cite{AFIT} form a kind of ideal hyperbolic analogue of the bicycle correspondence of \cite{CCM1}, making $S_\mathbf{P}$ analogous to the circumcenter of mass. Yet another infamous polygonal center is the fixed point of the pentagram map, which was discovered in \cite{Schwartz} and coordinatized in \cite{Glick}. In \cite[\S1]{Izosimov} it is suggested that this point may be the limit point as $\varepsilon\rightarrow 0$ of an eigenvector for the infinitesimal monodromy of a deformation into twisted polygons $\mathbf{P}_{1+\varepsilon}$ of a closed polygon $\mathbf{P}_1$  (cf.~Remark~\ref{rm:alpha-related}). And yet, it remains unknown whether the circumcenter of mass or the fixed point of the pentagram map is defined by an optimality property analogous to that of the hyperbolic barycenter in Theorem~\ref{th:Klei}. We refer to \cite{KAISER} for a systematic discussion of polygonal centers defined by optimality conditions.
\end{rmk}

\section{Case studies: small-gons}\label{sec:concurrence}

The Interpolation Lemma and the Archimedian Property (and their corollaries) established in \S\ref{sec:polygons} 
are especially useful for geometrically constructing the hyperbolic barycenters of ideal hyperbolic $n$-gons recursively
. As we have already extensively addressed the particular case of triangles, let us now illustrate these constructions in the next few particular cases of $n=4,5,6,10$, where we begin to glimpse a cascade of concurrences culminating at the hyperbolic barycenter.

\subsection{Case study: Quadrilaterals}\label{sec:quads}

A first construction of the hyperbolic barycenter of an ideal quadrilateral $\mathbf{P}$ (and in general, of an ideal polygon) is achieved via the Archimedian Property recursively from the case of triangles (and in general, from the case of smaller-gons). Let $S_\ell$ be the symmedian point of the triangle $\Delta (P_{\ell-1}P_\ell P_{\ell+1})$. Then by Corollary~\ref{cor:archimedian} $S_\mathbf{P}$ lies on $S_1S_3$ and on $S_2S_4$ (see Figure \ref{fig:quad}). 

Surprisingly, an alternative construction makes it easier to construct hyperbolic barycenters of ideal quadrilaterals than those of ideal triangles, as we see in the next result.

\begin{thm}\label{th:quad-barycenter}
    The hyperbolic barycenter of an ideal quadrilateral $\mathbf{P}$ lies on its diagonals. 
\end{thm}

\begin{proof} Let $\Xi_\ell=\{\xi_{\ell-1},\xi_\ell\}$ be the set of sides meeting at the ideal vertex $P_\ell$, where as usual we take indices modulo $n=4$. With notation as in \eqref{eq:xi-notation}, we see that each $S_{\Xi_\ell}=P_\ell$ (cf.~Remark~\ref{rm:twosidesminimizer}), and \[\mathcal{Q}_{\Xi_1}(X)+\mathcal{Q}_{\Xi_3}(X)=\mathcal{Q}_\mathbf{P}(X)=\mathcal{Q}_{\Xi_2}(X)+\mathcal{Q}_{\Xi_4}(X).\] By Corollary~\ref{rm:metathm}, $S_{\{\Xi_1,\Xi_3\}}=S_\mathbf{P}=S_{\{\Xi_2,\Xi_4\}}$ lies on $P_1P_3$ and on $P_2P_4$.\qedhere
\end{proof}

\begin{figure}[hbt!]
    \centering
        \includegraphics[width=0.6\textwidth]{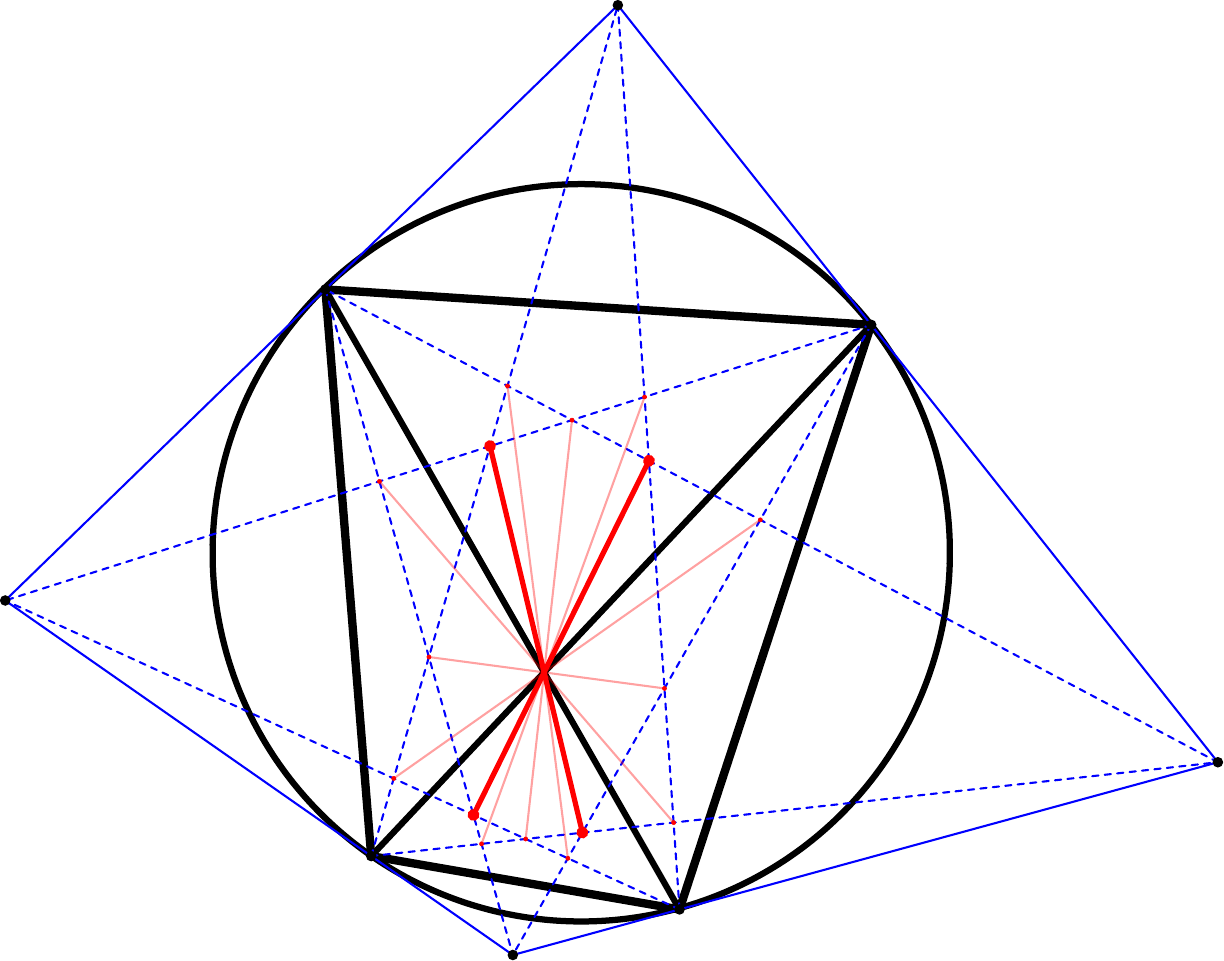}
    \caption{The hyperbolic barycenter of an ideal quadrilateral is the intersection point of its diagonals as well as the intersection of the segments connecting symmedians of opposite triangles in each triangulation. }
    \label{fig:quad}
\end{figure}

\begin{cor}\label{cor:big-diagonals} The hyperbolic barycenter of an ideal quadrilateral $\mathbf{P}$ is the intersection of the diagonals $P_1^*P_3^*$ and $P_2^*P_4^*$, of the quadrilateral formed by the polars to the sides.
\end{cor}

\begin{proof}
    This is an immediate consequence of Theorem~\ref{th:quad-barycenter} and the well-known (Euclidean!) fact \cite[11.1.7]{aakop} that the diagonals of a circumscribed quadrilateral intersect at the same point as the diagonals of the inscribed quadrilateral whose vertices are the points of tangency.
\end{proof}

The trivial observation in Corollary~\ref{cor:big-diagonals} has the non-trivial consequence in Lemma~\ref{lm:two-askew-sides} below that the hyperbolic barycenter of an ideal quadrilateral simultaneously and independently minimizes the sum of hyperbolic sines of hyperbolic distances to each pair of opposite sides.

\begin{lm} \label{lm:two-askew-sides}
    For $\mathbf{P}$ a convex ideal hyperbolic quadrilateral with sides $\xi_\ell=P_\ell P_{\ell+1}$, let $\,\Xi_o=\{\xi_1,\xi_3\}$ and $\Xi_e=\{\xi_2,\xi_4\}$. Then $S_{\Xi_o}=S_\mathbf{P}=S_{\Xi_e}$.
\end{lm}

\begin{proof}
    Writing $\boldsymbol{\Xi}_{\text{odd}}=\{\{\xi_1\},\{\xi_3\}\}$ and $\boldsymbol{\Xi}_\text{even}=\{\{\xi_2\},\{\xi_4\}\}$, we conclude from Lemma~\ref{lm:interpolation} that $S_{\boldsymbol{\Xi}_\text{odd}}=S_{\Xi_o}$ and $S_{\boldsymbol{\Xi}_\text{even}}=S_{\Xi_e}$, whence by Corollary~\ref{rm:metathm} $S_{\Xi_o}$ lies on $P_1^*P_3^*$ and $S_{\Xi_e}$ lies on $P_2^*P_4^*$ (cf.~Remark~\ref{rm:twosidesminimizer}). Letting $\boldsymbol{\Xi}_\text{tot}=\{\Xi_o,\Xi_e\}$, then $S_{\boldsymbol{\Xi}_\text{tot}}=S_\mathbf{P}$ again by Lemma~\ref{lm:interpolation}, whence again by Corollary~\ref{rm:metathm} $S_\mathbf{P}$ lies on the geodesic segment $S_{\Xi_o}S_{\Xi_e}$. Since $S_\mathbf{P}$ is the intersection of $P_1^*P_3^*$ and $P_2^*P_4^*$ by Corollary~\ref{cor:big-diagonals}, this forces $S_{\Xi_o}$ to also lie on $P_2^*P_4^*$ and $S_{\Xi_e}$ to also lie on $P_1^*P_3^*$.
\end{proof}

    Besides the two constructions of the hyperbolic barycenter $S_\mathbf{P}$ of an ideal quadrilateral $\mathbf{P}$, we witness in Figure \ref{fig:quad} many additional concurrences at $S_\mathbf{P}$, which we leave unaddressed as exercises for the interested reader.

\subsection{Case study: Pentagons}

Our first construction of the hyperbolic barycenter of an ideal pentagon $\mathbf{P}$ is obtained from the Archimedian Property. Let $S_\ell$ be the symmedian point of the triangle $\Delta (P_{\ell-1}P_\ell P_{\ell+1})$, and let $T_\ell$ be the intersection of the two diagonals not containing $P_\ell$: $P_{\ell-1}P_{\ell+2}$ and $P_{\ell+1}P_{\ell-2}$. Then by Theorem~\ref{th:quad-barycenter}, $T_\ell$ is the hyperbolic barycenter of the quadrilateral $P_{\ell-2}P_{\ell-1}P_{\ell+1}P_{\ell+2}$. Thus by Corollary~\ref{cor:archimedian}, the five geodesic segments $S_\ell T_\ell$ are concurrent at $S_\mathbf{P}$.

On the other hand, we can use the Interpolation Lemma to provide an alternative construction of the hyperbolic barycenter $S_\mathbf{P}$.

 \begin{lm} \label{lm:pentagon2}
    Let $\mathbf{P}$ be a convex ideal hyperbolic pentagon. Let $R_{\ell}$ be the intersection of the two hyperbolic altitudes, from the vertex $P_{\ell+2}$ to the side $P_{\ell-2}P_{\ell-1}$, and from the vertex $P_{\ell-2}$ to the side $P_{\ell+1}P_{\ell+2}$, i.e., $R_\ell$ lies on $P_{\ell+2}P_{\ell+3}^*$ and on $P_{\ell+3}P_{\ell+1}^*$. Then the hyperbolic barycenter $S_{\mathbf{P}}$ lies on the geodesic segment $P_\ell R_\ell$.
 \end{lm}
 
\begin{proof}
    By Corollary~\ref{rm:metathm} applied to $\boldsymbol{\Xi}=\{\Xi_{1},\Xi_2\}$, with $\Xi_1=\{\xi_{\ell-1},\xi_\ell\}$ and $\Xi_2=\{\xi_{\ell+1},\xi_{\ell+2},\xi_{\ell+3}
    \}$, $S_{\boldsymbol{\Xi}}=S_{\mathbf{P}}$ lies on $S_{\Xi_1}S_{\Xi_2}$. Now $S_{\Xi_1}=P_\ell$ (cf.~Remark~\ref{rm:twosidesminimizer}), and $S_{\Xi_2}=R_\ell$ by Lemma~\ref{lm:three-sides-minimizer}. 
\end{proof}

We witness in the case of ideal pentagons also, as in the case of ideal quadrilaterals, the hyperbolic barycenter as the epicenter of several kinds of concurrences.

\begin{figure}[!hbt]
    \centering
    \includegraphics[width=0.4 \textwidth]{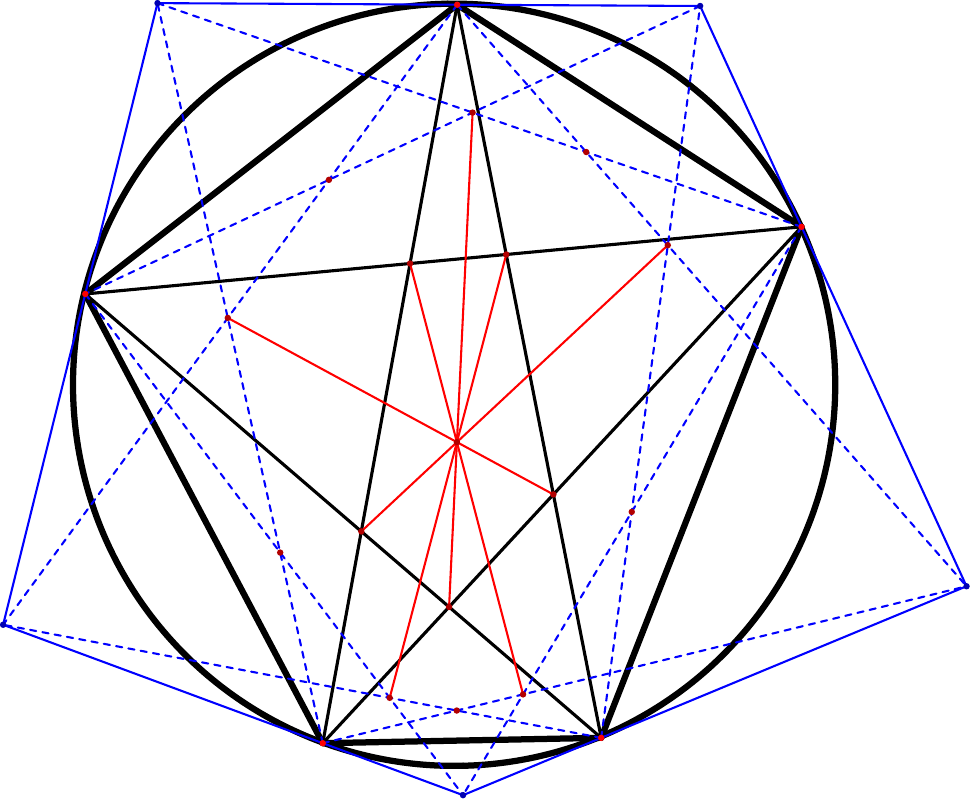} \hspace{.4in}\includegraphics[width=0.4 \textwidth]{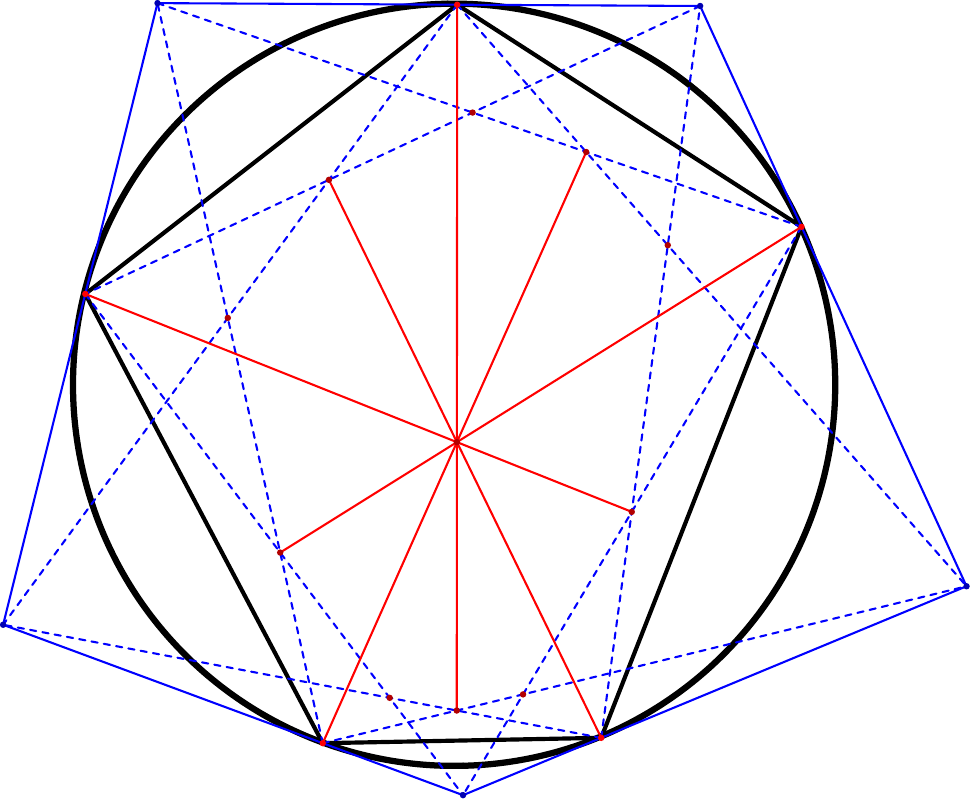}
    \caption{Constructions of the hyperbolic barycenter of an ideal pentagon. Left: via the Archimedian property. Right: via the Interpolation Lemma.}
    \label{fig:SymNgon}
\end{figure}

\subsection{Case study: Hexagons} Surprisingly, it is relatively easier to construct the hyperbolic barycenter of an ideal hexagon than that of an ideal pentagon. In fact, the construction of the former arises from one of the variants of the Brianchon Theorem (see \cite[11.1.9]{aakop} and Figure~\ref{fig:hex}).

\begin{thm}\label{th:hex}
    Let $\mathbf{Q}$ be the hexagon in $\mathbb{H}$ whose sides lie on the the short diagonals of the ideal hexagon $\mathbf{P}$. Then the long diagonals of $\,\mathbf{Q}$ all intersect at $S_\mathbf{P}$.
\end{thm}

\begin{proof}
For each $\ell$, consider the ideal quadrilateral $(p_{\ell-1},p_\ell,p_{\ell+1},p_{\ell+2})$, and let $Q_\ell$ be its hyperbolic barycenter. By Theorem~\ref{th:quad-barycenter}, $Q_\ell$ is the intersection point of $P_\ell P_{\ell+2}$ and $P_{\ell-1}P_{\ell+1}$, and by Corollary~\ref{lm:arch}, $S_\mathbf{P}$ lies on the geodesic segment $Q_\ell Q_{\ell+3}$.\qedhere
\end{proof}

Another construction of the hyperbolic barycenter of an ideal hexagon is obtained from the Interpolation Lemma, as follows. Let $R_\ell$ be the intersection of $P_\ell P_{\ell+1}^*$ and $P_{\ell+1}P_{\ell-1}^*$. Then $R_\ell=S_{\Xi_\ell}$ for $\Xi_\ell:=\{\xi_{\ell-1},\xi_\ell,\xi_{\ell+1}\}$ by Lemma~\ref{lm:three-sides-minimizer}. Applying Corollary~\ref{rm:metathm} to $\boldsymbol{\Xi}=\{\Xi_\ell,\Xi_{\ell+3}\},$ we find that $S_{\mathbf{P}}$ lies on each geodesic segment $R_\ell R_{\ell+3}$.

\begin{figure}[!hbt]
    \centering
    \includegraphics[width=0.4 \textwidth]{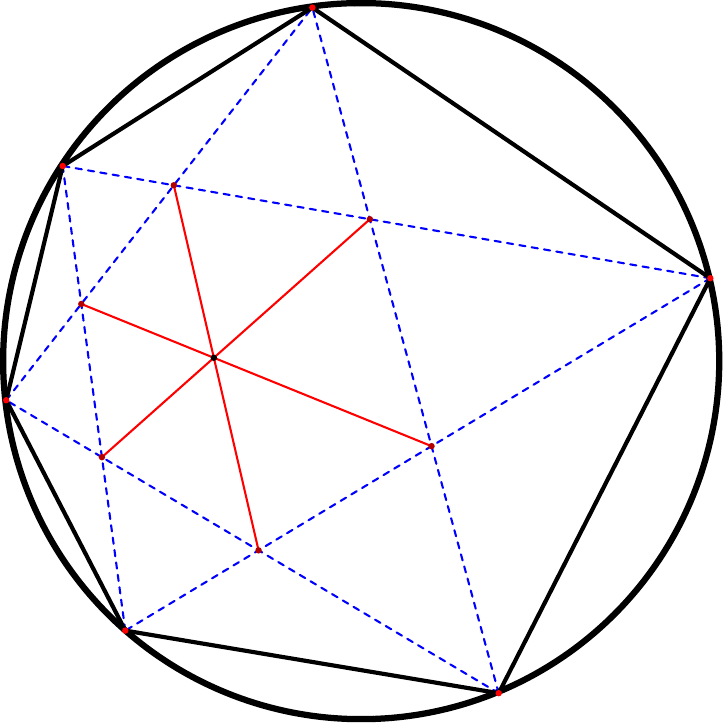}\hspace{.4in}
    \includegraphics[width=0.4 \textwidth]{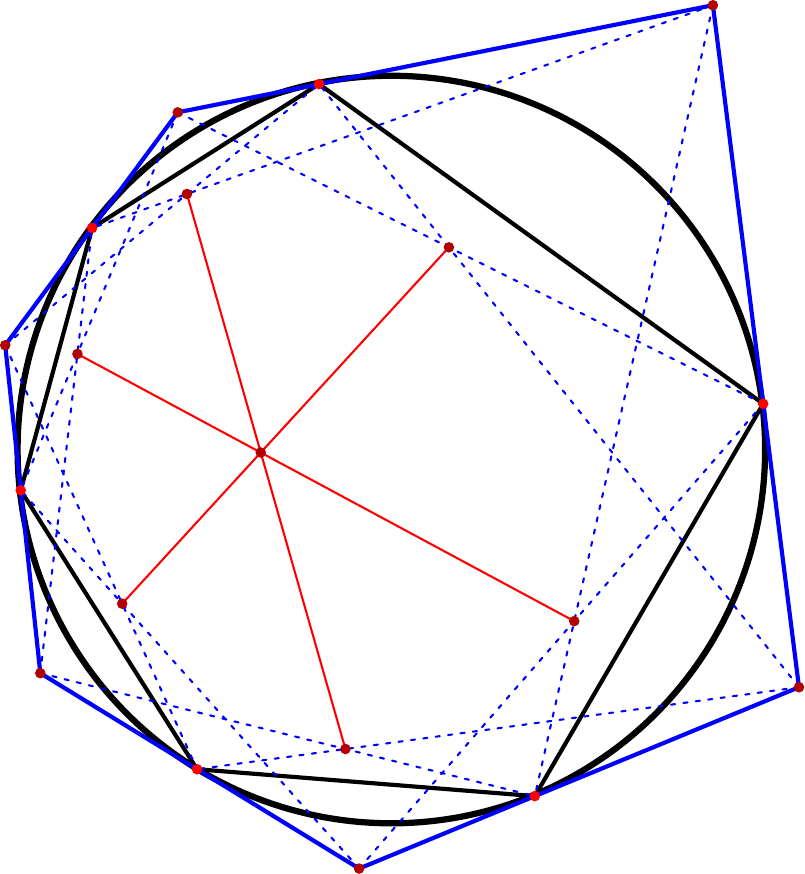}
    \caption{Constructions of the hyperbolic barycenter of an ideal hexagon. Left: via the Archimedian property. Right: via the Interpolation Lemma.}
    \label{fig:hex}
\end{figure}

\subsection{Case study: $(2^m+2)$-gons}

The (initially surprising) relative simplicity in the construction of hyperbolic barycenters for ideal hexagons versus ideal pentagons is part of a general phenomenon, which leads to a layered family of interesting concurrences arising from the Archimedian Property. We can subdivide an ideal $(2^m+2)$-gon $\mathbf{P}$ along any of its $2^{m-1}+1$ long diagonals into two $(2^{m-1}+2)$-gons \mbox{$\mathbf{Q}_1$ and $\mathbf{Q}_2$,} and conclude by Corollary~\ref{cor:archimedian} that $S_{\mathbf{P}}$ lies on $S_{\mathbf{Q}_1}S_{\mathbf{Q}_2}$.  This gives a concurrence of $2^{m-1}+1$ lines, each of whose endpoints is itself a concurrence of $2^{m-2}+1$ lines, each of whose endpoints is itself a concurrence \ldots, etc. This recursively reduces the construction of $S_\mathbf{P}$ to any one of the many equivalent constructions of the hyperbolic barycenter of an ideal quadrilateral in \S\ref{sec:quads} (see Figure~\ref{fig:quad}).
In Figure~\ref{fig:decagon} we depict the next case $m=3$ of this construction, to obtain the hyperbolic barycenter of an ideal decagon as the point of concurrence of the five line segments connecting opposite hyperbolic barycenters of ideal hexagons constructed according to Theorem~\ref{th:hex}.

\begin{figure}[!hbt]
    \centering
    \includegraphics[width=0.6\textwidth]{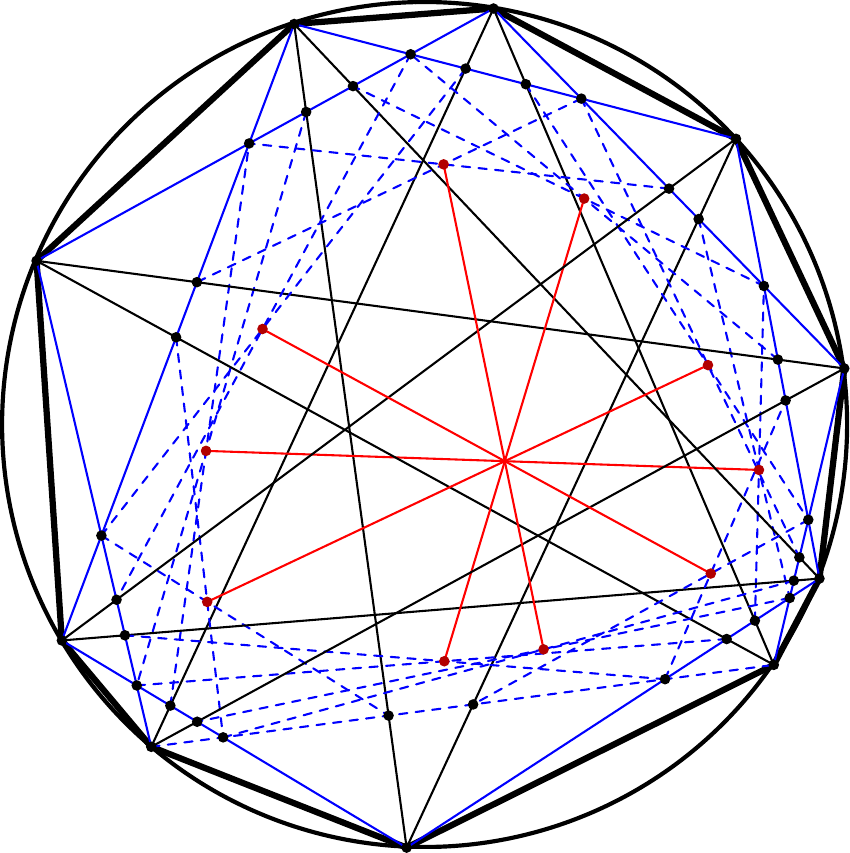}
    \caption{Recursive construction of the hyperbolic barycenter of an ideal decagon. 
    }
    \label{fig:decagon}
\end{figure}

  \section{Moduli of ideal polygons with fixed hyperbolic barycenter}\label{sec:moduli}

Let us denote by $\mathcal{M}_n(S)=\mathcal{M}_n(\alpha,\beta)$ the moduli space of ideal $n$-gons having a fixed $S=(\alpha,\beta)\in\mathbb{H}
$ as their 
hyperbolic barycenter. We immediately see from the coordinatization of $S_\mathbf{P}$ in \eqref{eq:Sym} that \begin{equation}\label{eq:moduli}\begin{cases}I_\mathbf{P}-K_\mathbf{P}=\alpha\cdot(I_\mathbf{P}+K_\mathbf{P});\\ 2J_\mathbf{P}=\beta\cdot(I_\mathbf{P}+K_\mathbf{P})\end{cases}\end{equation} defines a subspace of codimension at most $2$ in the moduli space $\mathcal{P}_n$ of ideal hyperbolic $n$-gons. We would like to have a more geometrically meaningful description of $\mathcal{M}_n(S)$ than merely the algebraic conditions in \eqref{eq:moduli}. Here we accomplish this only for the smallest values of $n=3,4$. 

\subsection{Moduli of ideal triangles with fixed hyperbolic barycenter}

The space $\mathcal{P}_3$ of all ideal hyperbolic triangles is three-dimensional, hence the subspaces $\mathcal{M}_3(S)$ are expected to be one-dimensional. It is clear from the minimality property of $S_\mathbf{P}$ in Theorem~\ref{th:Klei} that all the triangles obtained from $\mathbf{P}$ by a hyperbolic rotation about $S_\mathbf{P}$ share the same hyperbolic barycenter, yielding a one-dimensional family, as expected. It is easy to see that this is indeed all of $\mathcal{M}_3(S_\mathbf{P})$, as we show next.

\begin{lm}\label{lm:triangle-rotation}
    Two ideal triangles $\mathbf{P}_1$ and $\mathbf{P}_2$ share the same hyperbolic barycenter $S_{\mathbf{P}_1}=S=S_{\mathbf{P}_2}$ if and only if they are obtained from one another by a hyperbolic rotation about $S$.
\end{lm} 
    \begin{proof} It suffices to show that if an ideal triangle $\mathbf{P}=(p_1,p_2,p_3)$ has $S_\mathbf{P}=(0,0)$ then it is regular. This follows from a simple algebraic computation. If we rotate such a triangle so that one of its vertices is at $0\in\mathbb{R}\mathbb{P}^1$, say $p_3=0$, then $2J_\mathbf{P}=0$ if and only if $p_1+p_2=0$, and $I_\mathbf{P}-K_\mathbf{P}=0$ if and only if $p_1^2p_2^2-p_1p_2-(p_1-p_2)^2=0$, which together imply $\{p_1,p_2\}=\{\pm\sqrt{3}\}$, as required. 
    \end{proof}

An alternative construction of $\mathcal{M}_3(S)$ is as follows. Given an ideal hyperbolic triangle $\mathbf{P}$, consider the ideal hyperbolic triangle $\hat{\mathbf{P}}$, whose ideal vertex $\hat{P}_\ell$ is the hyperbolic reflection of $P_\ell$ with respect to its opposite side $\xi_{\ell+1}=P_{\ell+1}P_{\ell+2}$. Then the points $P_\ell$, $\hat{P}_\ell$, and $P_{\ell+1}^*$ are collinear by construction. Since $S_\mathbf{P}$ is the hyperbolic orthocenter of $\mathbf{P}$, it also lies on $P_\ell P_{\ell+1}^*$, and by symmetry (since $\hat{\hat{\mathbf{P}}}=\mathbf{P}$) we see that $\hat{P}_{\ell+1}^*$ is also on this line. Hence $S_\mathbf{P}$ is the hyperbolic orthocenter of $\hat{\mathbf{P}}$.

\begin{lm} \label{lm:six-points}
    The six polars $P_1^*,P_2^*,P_3^*, \hat{P}_1^*, \hat{P}_2^*, \hat{P}_3^*$ all lie on the same conic $\Gamma_\mathbf{P}$.
\end{lm}

\begin{proof}
    Consider the unique conic $\Gamma$ passing through the first five of these points, except possibly $\hat{P}_3^*$. Then we see that $P_1^*,P_2^*,P_3^*$ form a $3$-periodic Poncelet trajectory on $\Gamma$ circumscribed about $\mathbf{O}$. By the Poncelet Theorem \cite[Thm.~2.14]{Dragovic}, the Poncelet trajectory on $\Gamma$ containing $\hat{P}_1^*$ and $\hat{P}_2^*$ is also $3$-periodic, whence $\hat{P}_3^*$ also lies on $\Gamma$.
\end{proof}

\begin{thm}\label{th:triangle-moduli}
    The moduli space $\mathcal{M}_3(S_\mathbf{P})$ consists of the ideal triangles whose ideal vertices are the tangency points of polar triangles circuminscribed in the $3$-periodic Poncelet pair $(\Gamma_\mathbf{P},\mathbf{O})$, where $\Gamma_\mathbf{P}$ is the conic associated in Lemma~\ref{lm:six-points} with $\mathbf{P}$. The conic $\Gamma_\mathbf{P}=\Gamma(S_\mathbf{P})$ depends only on $S_\mathbf{P}$, \mbox{and not on $\mathbf{P}$.}
\end{thm}

\begin{proof}
    After applying an isometry sending $S$ to $(0,0)$, it suffices to prove this for the special case $S=(0,0)$, which is an immediate consequence of Lemma~\ref{lm:triangle-rotation}.
\end{proof}

In order to find $\Gamma(S)$ explicitly it suffices to find $\Gamma(r,0)$ for $0\leq r<1$, and then rotate by $\theta$ to reach any desired $S=(r\cos\theta,r\sin\theta)\in\mathbb{H}$. There are two special triangles $\mathbf{P}$ and $\hat{\mathbf{P}}$ such that $S_\mathbf{P}=S_{\hat{\mathbf{P}}}=(r,0)$ with $p_3=0$ and $\hat{p}_3=\infty$, respectively. We find $p_1+p_2=0=\hat{p}_1+\hat{p}_2$  (from $J=0$), and then that $p_1=\pm\sqrt{3\frac{1-r}{1+r}}$ and $ \hat{p}_1 =\frac{-p_1}{3}$ (from $I-J=r(I+J)$). We then compute from \eqref{eq:polar} that
\[
    P_1^*=\left(\tfrac{r-2}{1-2r},0\right);\ \  P_2^*=(1,-p_1);\ \  P_3^* = (1,p_1); \ \ 
    \hat{P}_1^* = \left(\tfrac{r+2}{2r+1},0\right);\ \  \hat{P}_2^* = (-1,3p_1^{-1}) ; \ \  \hat{P}_3^* = (-1,-3p_1^{-1}).
\]
It is trivial to verify that \[\Gamma(r,0)\ : \ \bigl(1-4r^2\bigr)x^2 + 6rx +\bigl(1-r^2)y^2+r^2-4 = 0\] is the unique (overdetermined!) conic passing through all six polars as in Lemma~\ref{lm:six-points}. The following amusing observation is then immediate.
\begin{cor}
    The Poncelet conic $\Gamma(S)$ is an ellipse (resp., a parabola; resp., a hyperbola) if and only if $|S|<1/2$ (resp., $|S|=1/2$; resp., $|S|>1/2$), where $|S|$ denotes the Euclidean norm of $S$.
\end{cor}

\begin{figure}[!hbt]
    \centering
    \includegraphics[width=0.47\textwidth] {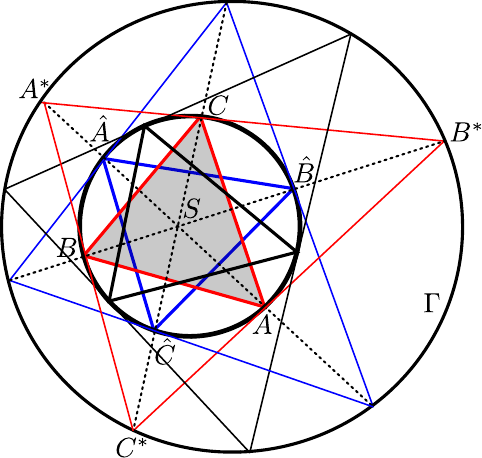} \hspace{.175in}\includegraphics[width=0.47\textwidth]{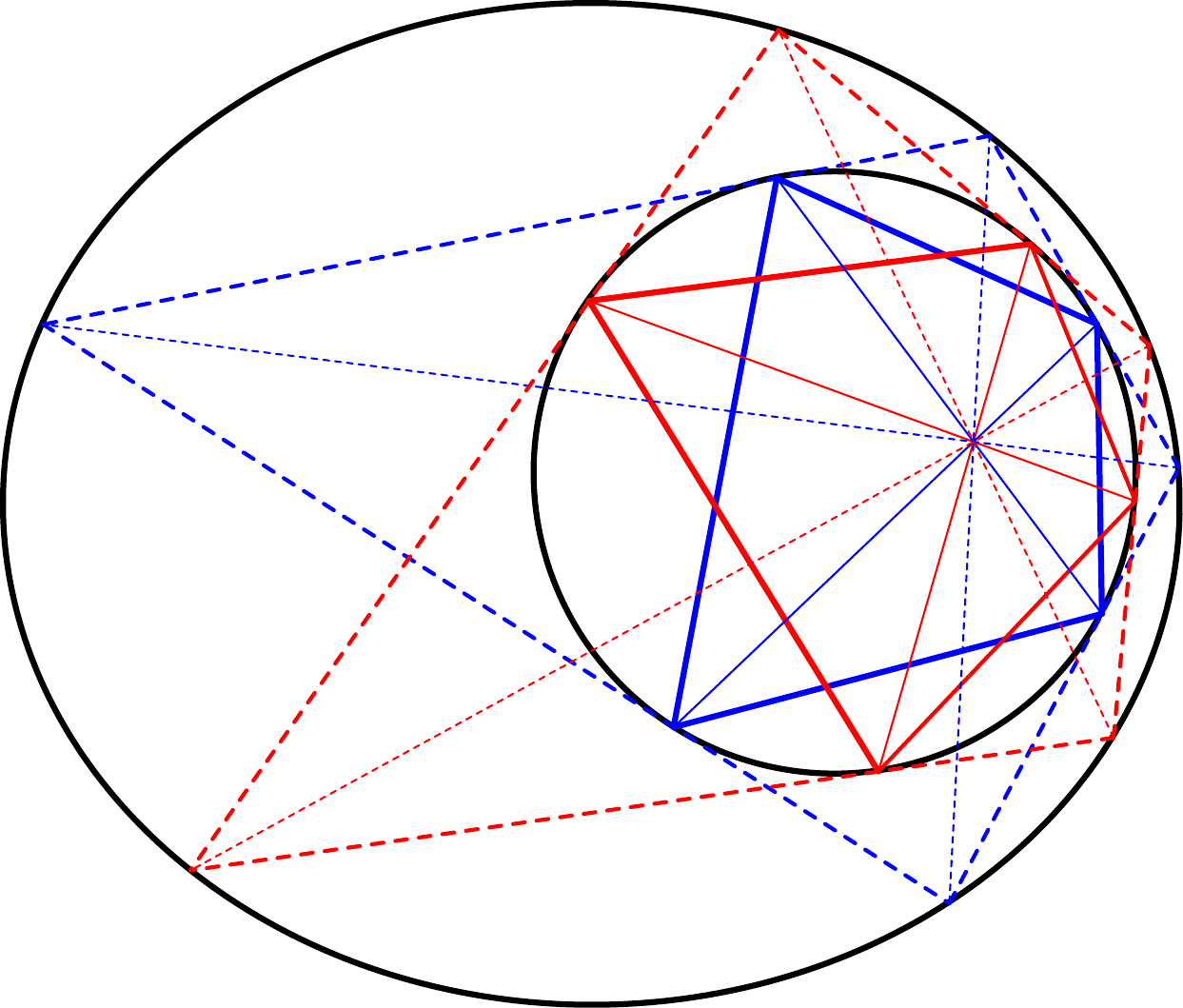}
    \caption{ Left: Poncelet conic of all ideal triangles with a common hyperbolic barycenter. Right: one representative of the pencil of Poncelet conics of ideal quadrilaterals with a common hyperbolic barycenter.}
    \label{fig:Iter}
\end{figure}

\subsection{Moduli of ideal quadrilaterals with fixed hyperbolic barycenter}

The space $\mathcal{P}_4$ of all ideal hyperbolic quadrilaterals is four-dimensional, hence the subspaces $\mathcal{M}_4(S)$ are expected to be two-dimensional. Theorem~\ref{th:quad-barycenter} immediately yields a compact and satisfying description of $\mathcal{M}_4(S)$: choose any two distinct hyperbolic lines $\delta_1$ and $\delta_2$ intersecting at $S$ to be the diagonals of the arbitrary ideal quadrilateral $\mathbf{P}$ with $S_\mathbf{P}$. This construction exhausts all of $\mathcal{M}_4(S)$, and is roughly analogous to our first description of $\mathcal{M}_3(S)$ in terms of rotations about~$S$ in Lemma~\ref{lm:triangle-rotation}. Inspired by the description of $\mathcal{M}_3(S)$ in terms of Poncelet conics afforded by Theorem~\ref{th:triangle-moduli}, we seek a similar description of $\mathcal{M}_4(S)$. 

Given an ideal hyperbolic quadrilateral $\mathbf{P}$, its four polars $P_\ell^*$ determine a unique pencil of conics $\boldsymbol{\Gamma}_\mathbf{P}$ passing through these four points. For a generic fifth point $Q$, let $\Gamma_\mathbf{P}(Q)$ be the unique conic on this pencil that also passes through $Q$. Then $(\Gamma_\mathbf{P}(Q),\mathbf{O})$ forms a $4$-periodic Poncelet pair, whence by the Poncelet Theorem \cite[Theorem~2.14]{Dragovic} $Q=Q_1^*$ is polar to the side of a unique ideal quadrilateral $\mathbf{Q}$ all of whose polars $Q_\ell^*$ are also on $\Gamma_\mathbf{P}(Q)$. It is not obvious (and for us, not even expected) that we should always have $S_\mathbf{Q}=S_\mathbf{P}$ as apparently depicted in Figure~\ref{fig:Iter}. In the next result we show that this is indeed the case, and that this construction exhausts all of $\mathcal{M}_4(S_\mathbf{P})$.

\begin{thm}\label{th:quad-moduli} The moduli space $\mathcal{M}_4(S_\mathbf{P})$ consists of the ideal quadrilaterals whose ideal vertices are the tangency points of polar quadrilaterals circuminscribed in a $4$-periodic Poncelet pair $(\Gamma,\mathbf{O})$, for some conic $\Gamma$ in the pencil $\boldsymbol{\Gamma}_\mathbf{P}$.
\end{thm}

\begin{proof}
    After applying an isometry sending $S$ to $(0,0)$, it suffices to prove this for the special case $S=(0,0)$. By Corollary~\ref{cor:big-diagonals}, it suffices to prove the claim that for any $\mathbf{P}\in\mathcal{M}_4(0,0)$, any $\Gamma$ in $\boldsymbol{\Gamma}_\mathbf{P}$, and any $\mathbf{Q}^*$ circuminscribed in $(\Gamma,\mathbf{O})$, the diagonals of $\mathbf{Q}^*$ intersect at the origin.
    
    To prove the claim, first note that $S_\mathbf{P}=(0,0)$ if and only if $\mathbf{P}$ is a rectangle, which after a harmless rotation about the origin may be assumed to have its sides parallel to the coordinate axes. Then the polar quadrilateral $\mathbf{P}^*=P_1^*P_2^*P_3^*P_4^*$ is a rhombus whose vertices lie on the coordinate axes. By comparing the general form of a conic evaluated at each pair $P_\ell^*$ and $P_{\ell+2}^*$, we discover immediately that every (non-degenerate) $\Gamma$ in the pencil $\boldsymbol{\Gamma}_\mathbf{P}$ must be a central ellipse. Our claim follows immediately from the following general Lemma~\ref{lm:quad-orbit}.\qedhere
\end{proof}

We have severed the following configuration result from the rest of the proof of Theorem~\ref{th:quad-moduli} due to its independent interest. We provide a short conceptual proof because we have not found a reference for it in the literature, even though it is surely well-known to the experts.

\begin{lm}\label{lm:quad-orbit}
    If $\Gamma$ is a central conic such that $(\Gamma,\mathbf{O})$ forms a $4$-periodic Poncelet pair then every circuminscribed polygon in $(\Gamma,\mathbf{O})$ is a rhombus centered at the origin.
\end{lm}

\begin{proof}
    Let us first show this in the special case where our circuminscribed $\mathbf{Q}$ has one of its vertices, call it $Q_1$, on a vertex of $\Gamma$. Begin constructing the Poncelet orbit in both directions starting from $Q_1$, to obtain $Q_2$ and $Q_4$. By symmetry, these are both the next vertices of the ellipse (as we claim), or both in the hemisphere containing $Q_1$, or both in the hemisphere containing $-Q_1$. We see again by symmetry that the latter two impossibilities would contradict the $4$-periodicity of the Poncelet pair $(\Gamma,\mathbf{O})$.

    We can assume without loss of generality that the axes of $\Gamma$ coincide with the coordinate axes. We ignore from now on the special case already established above. Let us label the vertices of any other $\mathbf{Q}$ circuminscribed in $(\Gamma,\mathbf{O})$ such that $Q_\ell$ lies on the $\ell$-th quadrant. For any other such $\hat{\mathbf{Q}}\neq \mathbf{Q}$ circuminscribed in $(\Gamma,\mathbf{O})$, the vertex $\hat{Q}_\ell$ in the $\ell$-th quadrant lies before (resp., after) $Q_\ell$, relative to the standard counterclockwise orientation of $\Gamma$, for all $\ell$ simultaneously. Suppose that $Q_3$ lies before (resp., after) $-Q_1$ for some $\mathbf{Q}$, contrary to our contention. Then every vertex of $\mathbf{Q}$ lies before (resp., after) every vertex of $-\mathbf{Q}$. But this would imply, by symmetry, that every vertex of $-\mathbf{Q}$ must lie before (resp., after) every vertex of $-(-\mathbf{Q})=\mathbf{Q}$, which is absurd.
\end{proof}

\begin{rmk}
    We see from Theorem~\ref{th:quad-moduli} that for any two ideal quadrilaterals $\ \mathbf{P}\neq \mathbf{Q}$ such that $S_\mathbf{P}=S_\mathbf{Q}$, the eight distinct polars of $\mathbf{P}$ and $\mathbf{Q}$ belong to a conic, which is in fact the unique common conic to both pencils $\boldsymbol{\Gamma}_\mathbf{P}$ and $\boldsymbol{\Gamma}_\mathbf{Q}$.
 \end{rmk}

    In the case of triangles, $\mathcal{M}_3(S_\mathbf{P})$ is triply covered by the conic $\Gamma_\mathbf{P}$ traced out by the polars of the triangles obtained by rotating $\mathbf{P}$ about $S_\mathbf{P}$. This exceptional phenomenon does not occur in general for ideal quadrilaterals $\mathbf{P}$, with an important exception. After an isometry sending $S_\mathbf{P}$ to the origin, it we see that the polars of the rotations of $\mathbf{P}$ about $S_\mathbf{P}$ in general trace out two disjoint conics corresponding to pairs of opposite sides. These two conics coincide precisely when $\mathbf{P}$ is \emph{harmonic}, which case we begin to study systematically in the next section.

    \begin{rmk}
        The five polars of an ideal pentagon $\mathbf{P}$ define a unique conic $\Gamma$, which again by the Poncelet porism yields a one-parameter family of pentagons circuminscribed in $(\Gamma,\mathbf{O})$. Though it might seem natural to hope for all these pentagons to share the same hyperbolic barycenter $S_\mathbf{P}$, we have verified experimentally that they do not. We leave open the tantalizing question of how to describe, in a similarly geometric way as for ideal triangles and quadrilaterals, the corresponding moduli spaces for ideal hyperbolic $n$-gons for $n\geq 5$, beyond the obvious algebraic conditions \eqref{eq:moduli}.
    \end{rmk}

\section{Least-squares points and harmonic polygons}\label{sec:euclidean}

Thus far we have been mainly concerned with entities and phenomena occurring within hyperbolic geometry. In Theorem~\ref{th:Klei}, the hyperbolic barycenter $S_\mathbf{P}$ of an ideal triangle $\mathbf{P}$ is uniquely defined by the property that it minimizes the sum of hyperbolic sines of hyperbolic distances to the sides. Now the classical Theorem~\ref{th:Euc} states that the deKleinization of this same $S_\mathbf{P}$ happens to coincide with the \emph{least-squares point} $L_\mathbf{P}$, uniquely defined by the property that it minimizes the sum of squares of Euclidean distances to the sides. This strange and unexpected coincidence begs the natural question, for which other ideal polygons it might happen that $S_\mathbf{P}=L_\mathbf{P}$. 

In this section we keep considering ``the same'' $\mathbf{P}$ intermittently as an ideal hyperbolic polygon and as an inscribed Euclidean polygon, the first (resp., latter) being the (de)Kleinization of the latter (resp., former). To avoid cumbersome notation, we keep using the same symbol~$\mathbf{P}$, making it clear from context or explicitly each time which geometry we are working with.

An inscribed Euclidean (resp., ideal hyperbolic) polygon $\mathbf{P}$ is \emph{harmonic} if it is obtained from a regular polygon by a projective transformation (resp., hyperbolic isometry). These Euclidean and hyperbolic notions agree under the (de)Kleinization correspondence. Harmonic polygons have been extensively studied since the late 1800s -- see for example \cite{casey, simmons, tarry}, and the recent \cite{Reznik}.

As mentioned in the introduction, one of our initial motivations for this work was to find an explanation for the classical but strange Theorem~\ref{th:Euc}. Our main goal in this section is to provide a simple proof of Theorem~\ref{th:harmonic-coincidence}: for every harmonic polygon $\mathbf{P}$, the hyperbolic barycenter $S_\mathbf{P}$ and the least-squares point $L_\mathbf{P}$ coincide. From this we propose the provocative explanation: $S_\mathbf{P}=L_\mathbf{P}$ for all triangles $\mathbf{P}$ only because all triangles are harmonic.

\begin{rmk}
     The coordinates of the symmedian $S$ of a triangle have been obtained explicitly in \cite{Sadek}, in terms of the equations of the lines formed by the sides of the triangle, using the optimality property in Theorem~\ref{th:Euc} as a defining property of $S$. In principle, the coordinatization of \cite{Sadek} and ours in \eqref{eq:Sym} could be deduced directly from one another by a series of unenlightening algebraic manipulations. Theorem~\ref{th:harmonic-coincidence}, which contains Theorem~\ref{th:Euc} as its first special case, can be understood as a more conceptual explanation for the agreement of the two coordinatizations of the same point.
\end{rmk}

Denoting by $d_\ell(X)$ the Euclidean distance from the point $X$ to the line $\xi_\ell$, we compute that 
\begin{equation}
    \label{eq:EuD} d_\ell(X)=\dfrac{1}{|P_\ell^*|} \left(1\mp P_\ell^*\cdot X\right),
\end{equation}
where the sign depends on whether $X$ and $P_\ell^*$ are on opposite sides or on the same side of $\xi_\ell$, respectively.
For $X$ in the interior of $\mathbf{P}$, we then compute the gradient 
\begin{equation}
    \label{eq:EuBar}
    \nabla\left(\sum_\ell d_\ell^2(X)\right)=2\sum\limits_{\ell}\left(\dfrac{P^*_\ell\cdot X-1}{|P_\ell^*|^2}\right)P^*_\ell. 
\end{equation}
From now on we denote by $\sigma_\ell$ the Euclidean length of $\xi_\ell$.

For the following basic result we include an (unnecessary, but short) argument in the spirit of some of our other computations, to ensure non-circularity.

\begin{lm} \label{lm:euclidean-proportionality}
    The locus of points $X$ satisfying the Euclidean proportionality condition \begin{equation}\label{eq:d-sigma-proportion}[d_{\ell-1}(X) : \sigma_{\ell-1}] =[d_{\ell}(X) : \sigma_{\ell}]\end{equation} is the hyperbolic altitude from $P_\ell$ to its short diagonal $P_{\ell-1}P_{\ell+1}$. 
\end{lm}

\begin{proof}
This is an elementary fact about any triangle $P_1P_2P_3$.
Under the parametrization of \S\ref{sec:elements}, each $|P^*_t|=\sec\left(\frac{\varphi_t-\varphi_{t+1}}{2}\right)$ and $\sigma_t=2\sin\left(\frac{\varphi_t-\varphi_{t+1}}{2}\right)$. Obviously, $P_2$ is on the line $d_1(X)\sigma_2=d_2(X)\sigma_3$. That $P_3^*$ is also on this line is reduced by \eqref{eq:EuD} to the short computation  that \[(I_1+K_1)(I_3+K_3)-(I_1-K_1)(I_3-K_3)-4J_1J_3=1=(I_2+K_2)(I_3+K_3)-(I_2-K_2)(I_3-K_3)-4J_2J_3.\qedhere\]

\end{proof}

\begin{thm}\label{th:harmonic-orthocenter}
    A polygon $\mathbf{P}$ is harmonic if and only the hyperbolic altitudes from the vertices to their short diagonals are all concurrent. In this case, the point of concurrence is $S_\mathbf{P}$.
\end{thm}

\begin{proof}
    One direction is obvious: if $\mathbf{P}$ is harmonic then all the hyperbolic altitudes to short diagonals are concurrent at $S_\mathbf{P}$, since these entities are all preserved by hyperbolic isometry. 

    To prove the opposite implication, suppose the hyperbolic altitudes to the short diagonals of the $n$-gon $\mathbf{P}$ are concurrent at some $S$. After a hyperbolic isometry we can assume without loss of generality that $S=(0,0)$ and $P_\ell=\left(\cos\left(\frac{2\pi\ell}{n}\right),\sin\left(\frac{2\pi\ell}{n}\right)\right)$ for $\ell=1,2$. The remaining vertices are now recursively determined: 
    having already fixed $P_{\ell-1}$ and $P_{\ell}$, the short diagonal $P_{\ell-1}P_{\ell+1}$ is uniquely determined by its polar, which lies on the tangent line to $\mathbf{O}$ at $P_{\ell-1}$ and on $SP_\ell$. The vertices of the regular $n$-gon satisfy the same recurrence and initial conditions, so $\mathbf{P}$ is harmonic.
\end{proof}

\begin{rmk}
Our Euclidean definition of harmonicity, taken from \cite[\S4]{tarry}, is not the same as the one given in \cite[VI]{casey} and \cite[\S1]{simmons}, which asks for the existence of a point whose distances to the sides are proportional to the side lengths, called the symmedian of $\,\mathbf{P}$ in \cite[VI]{casey}. Mediated by Lemma~\ref{lm:euclidean-proportionality}, Theorem~\ref{th:harmonic-orthocenter} essentially states, from a hyperbolic barycentric perspective, that the two definitions of harmonicity agree. This is not a new result. The first implication in the above proof is essentially already contained in \cite[Thm.~4]{tarry}. A different proof of the opposite implication is given in \cite[\S27]{simmons} and attributed to Neuberg. Of course, these earlier results simply state that one of the definitions of harmonicity implies the other, and do not refer to hyperbolic altitudes at all. Theorem~\ref{th:harmonic-orthocenter} allows us to eponymously interpret all symmedians as hyperbolic barycenters. The former are defined only for harmonic polygons, whereas the latter are defined in \eqref{eq:Sym} for all ideal polygons.\end{rmk}

\begin{thm}\label{th:harmonic-coincidence}
   If $\,\mathbf{P}$ is harmonic then $S_\mathbf{P}=L_\mathbf{P}$.
\end{thm}

\begin{proof}
    Let $\boldsymbol{\sigma}:=(\sigma_1,\ldots,\sigma_n)$ denote the vector of Euclidean side-lengths of $\mathbf{P}$. For any $X$, let $\mathbf{d}(X):=(d_1(X),\ldots,d_n(X))$
    denote the vector of Euclidean distances from $X$
    to the sides of $\mathbf{P}$. Then $\mathbf{d}(X)\cdot \boldsymbol{\sigma}=2A$, twice the Euclidean area of $\mathbf{P}$, provided that $X$ is in the interior of $\mathbf{P}$. Now let $\boldsymbol{\sigma}^\perp(X):=\mathbf{d}(X)-\frac{2A}{|\boldsymbol{\sigma} |^2}\boldsymbol{\sigma}$ denote the orthogonal projection of $\mathbf{d}(X)$ onto the orthogonal complement of $\mathbb{R}\boldsymbol{\sigma}$. Then $|\mathbf{d}(X)|^2=4A^2|\boldsymbol{\sigma}|^{-2} + |\boldsymbol{\sigma}^\perp(X)|^2$. Therefore $L_\mathbf{P}$, which is by definition the minimizer of $|\mathbf{d}(X)|^2$, is also the minimizer of $|\boldsymbol{\sigma}^\perp(X)|^2$.
    
    In case $\mathbf{P}$ is harmonic, Theorem~\ref{th:harmonic-orthocenter} and 
    Lemma~\ref{lm:euclidean-proportionality} together imply that $\frac{d_{\ell-1}(S_\mathbf{P})}{\sigma_{\ell-1}}=\frac{d_\ell(S_\mathbf{P})}{\sigma_\ell}$ for each $\ell$. Therefore $\mathbf{d}(S_\mathbf{P})$ is a scalar multiple of $\boldsymbol{\sigma}$, whence $|\boldsymbol{\sigma}^\perp(S_\mathbf{P})|^2=0$ and $S_\mathbf{P}=L_\mathbf{P}$. \qedhere
\end{proof}

We see from the proof of Theorem~\ref{th:harmonic-coincidence} that $L_\mathbf{P}$ is equivalently characterized as the minimizer of $|\boldsymbol{\sigma}^\perp(X)|$, which achieves the minimum zero precisely when $\mathbf{P}$ is harmonic, by Theorem~\ref{th:harmonic-orthocenter}. In principle, one can express $L_\mathbf{P}$ algebraically in terms of $p_1,\dots,p_n$ by setting \eqref{eq:EuBar} to zero. The condition $S_\mathbf{P}=L_\mathbf{P}$ defines a subspace of codimension at most $2$ in the space $\mathcal{P}_n$ of all ideal $n$-gons (cf.~\S\ref{sec:moduli}). On the other hand, the subspace of $\mathcal{P}_n$ consisting of harmonic $n$-gons is three\nobreakdash-dimensional, and therefore the space of \mbox{$n$-gons} for which $S_\mathbf{P}=L_\mathbf{P}$ must strictly contain the space of harmonic $n$-gons for $n>5$. This can maybe be understood as some kind of ``regularity condition'' on $\mathbf{P}$, which is in general strictly weaker than harmonicity. And yet, according to the following result, harmonicity is equivalent to having $S_\mathbf{P}=L_\mathbf{P}$ in the next case $n=4$.

\begin{thm}\label{th:coincidence-4}
    If $\,\mathbf{P}$ is a quadrilateral, then $S_\mathbf{P}=L_\mathbf{P}$ if and only if $\,\mathbf{P}$ is harmonic or $S_\mathbf{P}=(0,0)$.
\end{thm}

\begin{proof}
    ($\Leftarrow$). It follows from Theorem~\ref{th:quad-barycenter} that $S_\mathbf{P}=(0,0)$ if and only if the Euclidean polygon $\mathbf{P}$ is a rectangle, in which case it is clear that $L_\mathbf{P}=(0,0)$. If $\mathbf{P}$ is harmonic then $S_\mathbf{P}=L_\mathbf{P}$ by Theorem~\ref{th:harmonic-coincidence}. 
    
    ($\Rightarrow$). Setting to zero the evaluation of \eqref{eq:EuBar} at $X=S_\mathbf{P}$, after some algebraic manipulation, 
    we conclude that either $I_\mathbf{P}-K_\mathbf{P}=0=J_\mathbf{P}$, or else the cross-ratio $[p_1,p_2,p_3,p_4]=1$. 
    These conditions are equivalent to having $S_\mathbf{P}=(0,0)$ and to $\mathbf{P}$ being harmonic, respectively.\qedhere
\end{proof}

\begin{rmk}
    The conclusion of Theorem~\ref{th:coincidence-4} may seem strange, because the origin is only special from the Euclidean point of view, but hyperbolically unremarkable. But the origin remains meaningful when comparing the two geometries, since the endofunctions of the unit disc which are simultaneously hyperbolic and Euclidean isometries is precisely the set of rotations about the origin.
\end{rmk}

We expect that a similar algebraic analysis should yield the following case $n=5$. One subtlety in this case is that the regular pentagram $\mathbf{P}$ still satisfies the proportionality condition $\boldsymbol{\sigma}^\perp(S_\mathbf{P})=\mathbf{0}$, and therefore so do all of its images under hyperbolic isometry, even though this $\mathbf{P}$ is not technically ``harmonic'' in our sense, and yet the subspace of $\mathcal{P}_5$ defined by $S_\mathbf{P}=L_\mathbf{P}$ does not care about the convexity of $\mathbf{P}$. We still optimistically hope for the best possible result.
\begin{cnj}\label{cnj:harmonic}
    If $\,\mathbf{P}=(p_1,\dots,p_5)$ is a pentagon, then $S_\mathbf{P}=L_\mathbf{P}$ if and only if either the five cross-ratios $[p_\ell,p_{\ell+1},p_{\ell+2},p_{\ell+3}]$ are all the same or $S_\mathbf{P}=(0,0)$.
\end{cnj}

\bibliographystyle{plain}

\begin{thebibliography}{10}

\bibitem{Izosimov}
Quinton Aboud and Anton Izosimov.
\newblock {The Limit Point of the Pentagram Map and Infinitesimal Monodromy}.
\newblock {\em International Mathematics Research Notices}, 2022(7):5383--5397,
  09 2020.

\bibitem{akop}
A.V. Akopyan.
\newblock On some classical constructions extended to hyperbolic geometry.
\newblock {\em Matematicheskoe prosveshenie}, 3(13):155--170, 2009.

\bibitem{aakop}
A.V. Akopyan.
\newblock {\em Geometry in figures}.
\newblock CreateSpace Independent Publishing Platform, 2017.

\bibitem{AFIT}
Maxim Arnold, Dmitry Fuchs, Ivan Izmestiev, and Serge Tabachnikov.
\newblock Cross-ratio dynamics on ideal polygons.
\newblock {\em Int. Math. Res. Not. IMRN}, (9):6770--6853, 2022.

\bibitem{Sadek}
M.~Bani-Yaghoub, Noah~H. Rhee, and Jawad Sadek.
\newblock An algebraic method to find the symmedian point of a triangle.
\newblock {\em Math. Mag.}, 89(3):197--200, 2016.

\bibitem{BergerII}
Marcel Berger.
\newblock {\em Geometry. {II}}.
\newblock Universitext. Springer-Verlag, Berlin, 1987.
\newblock Translated from the French by M. Cole and S. Levy.

\bibitem{casey}
J.~Casey.
\newblock Supplementary chapter.
\newblock In {\em {A Sequel to the First Six Books of the Elements of Euclid
  (5th Edition)}}, pages 165--222. Dublin: Hodges, Figgis, \& Co., 1888.

\bibitem{Cox}
H.~S.~M. Coxeter.
\newblock {\em The Real Projective Plane}.
\newblock Springer, New York, 1993.

\bibitem{RES}
Kostiantyn Drach and Richard~Evan Schwartz.
\newblock A hyperbolic view of the seven circles theorem.
\newblock {\em The Mathematical Intelligencer}, 42(2):61--65, 2020.

\bibitem{Dragovic}
V.~Dragovi\'c' and M.~Radnovi\'c.
\newblock {\em Poncelet Porisms and Beyond: Integrable Billiards, Hyperelliptic
  Jacobians and Pencils of Quadrics}.
\newblock Frontiers in Mathematics. Springer, Basel, 2011.

\bibitem{Reznik}
Ronaldo Garcia, Dan Reznik, and Pedro Roitman.
\newblock New properties of harmonic polygons.
\newblock {\em Journal for Geometry and Graphics}, 26(2):217--236, 2022.

\bibitem{Glick}
Max Glick.
\newblock {The Limit Point of the Pentagram Map}.
\newblock {\em International Mathematics Research Notices}, 2020(9):2818--2831,
  05 2018.

\bibitem{Honsberger}
Ross Honsberger.
\newblock {\em Episodes in nineteenth and twentieth century {E}uclidean
  geometry}, volume~37 of {\em New Mathematical Library}.
\newblock Mathematical Association of America, Washington, DC, 1995.

\bibitem{KAISER}
M.J. Kaiser and T.L. Morin.
\newblock Characterizing centers of convex bodies via optimization.
\newblock {\em Journal of Mathematical Analysis and Applications},
  184(3):533--559, 1994.

\bibitem{Mackay}
John~Sturgeon Mackay.
\newblock Early history of the symmedian point.
\newblock {\em Proceedings of the Edinburgh Mathematical Society}, 11:92 --
  103, 1892.

\bibitem{Schwartz}
Richard Schwartz.
\newblock {The pentagram map}.
\newblock {\em Experimental Mathematics}, 1(1):71 -- 81, 1992.

\bibitem{simmons}
T.~C. Simmons.
\newblock A new method for the investigation of harmonic polygons.
\newblock {\em Proceedings of the London Mathematical Society}, XVIII:289--304,
  1886.

\bibitem{CCM1}
Serge Tabachnikov and Emmanuel Tsukerman.
\newblock On the discrete bicycle transformation.
\newblock {\em Publ. Math. Uruguay}, 14:201--220, 2013.

\bibitem{CCM2}
Serge Tabachnikov and Emmanuel Tsukerman.
\newblock {Circumcenter of Mass and Generalized Euler Line}.
\newblock {\em Discrete Comput. Geom.}, 51(4):815–836, 2014.

\bibitem{tarry}
G.~Tarry and J.~Neuberg.
\newblock Sur les polygones et les poly\`edres harmoniques.
\newblock {\em Association Fran\c{c}aise pour l'Avancement des Sciences},
  Compte Rendu de la 15e Session - Nancy 1886(Seconde Partie - Notes et
  M\'emoires):12--24, 1887.

\end{thebibliography}

\end{document}